\documentclass[12pt,english]{article}
\usepackage[T1]{fontenc}
\usepackage[latin9]{inputenc}
\usepackage{geometry}
\usepackage{color}
\usepackage{amsmath}
\usepackage{amsthm}
\usepackage{amssymb}
\usepackage{stmaryrd}

\makeatletter
\numberwithin{equation}{section}
\theoremstyle{remark}
\newtheorem{rem}{\protect\remarkname}[section]
\theoremstyle{definition}
\newtheorem{defn}{\protect\definitionname}[section]
\theoremstyle{plain}
\newtheorem{lem}{\protect\lemmaname}[section]
\theoremstyle{plain}
\newtheorem{thm}{\protect\theoremname}[section]
\theoremstyle{plain}
\newtheorem{prop}{\protect\propositionname}[section]

\date{}

\makeatother

\usepackage{babel}
\providecommand{\definitionname}{Definition}
\providecommand{\lemmaname}{Lemma}
\providecommand{\propositionname}{Proposition}
\providecommand{\remarkname}{Remark}
\providecommand{\theoremname}{Theorem}

\begin{document}
\title{Lipschitz-stability of Controlled Rough Paths and Rough Differential
Equations}
\author{H. Boedihardjo\thanks{Department of Statistics, University of Warwick, Coventry, CV4 7AL,
United Kingdom. Email: horatio.boedihardjo@warwick.ac.uk.}$\ $ and X. Geng\thanks{School of Mathematics and Statistics, University of Melbourne, Parkville
VIC 3010, Australia. Email: xi.geng@unimelb.edu.au.}}
\maketitle
\begin{abstract}
We provide an account for the existence and uniqueness of solutions
to rough differential equations under the framework of controlled
rough paths. The case when the driving path is $\beta$-Hölder continuous,
for $\beta>1/3$, is widely available in the literature. In its extension
to the case when $\beta\leqslant1/3,$ a main challenge and missing
ingredient is to show that controlled roughs paths are closed under
composition with Lipschitz transformations. Establishing such a property
precisely, which has a strong algebraic nature, is a main purpose
of the present article.
\end{abstract}

\section{Introduction}

Multidimensional stochastic differential equations (SDEs) of the form
\begin{equation}
\mathrm{d}Y_{t}^{i}=\sum_{j=0}^{d}V_{j}(Y_{t})\mathrm{d}X_{t}^{j},\quad Y_{0}=y,\label{eq:SDE}
\end{equation}
where $X_{t}^{0}=t$, $(X_{t}^{j})_{j=1}^{d}$ is a $d$-dimensional
Brownian motion, and $(V_{j})_{j=0}^{d}$ are smooth vector fields
on $\mathbb{R}^{n}$, has been frequently used for modelling in mathematical
physics and finance (cf. \cite{Oks13} and the references therein).
The case when $V_{j}=0$ for all $j\geqslant0$ corresponds to ordinary
differential equations (ODEs). The SDE (\ref{eq:SDE}) also has applications
in pure mathematics. For instance, the distribution of its solution
can be used to study some second order linear parabolic and elliptic
differential equations, leading to probabilistic proofs of celebrated
results in PDE theory such as Hörmander's theorem (cf. Malliavin \cite{Mal78}).

When using Picard's iteration to establish the existence and uniqueness
of solutions to (\ref{eq:SDE}), the convergence of the iteration
is established under the $L^{2}$-norm with respect to the Wiener
measure. Partly inspired by the conjectures of H. Föllmer, Lyons \cite{Lyo98}
developed a \textit{pathwise} approach to construct the integral against
the ``$\mathrm{d}X_{t}^{j}$'s'' and showed the pathwise well-posedness
of the SDE. Lyons' pathwise estimates were performed through considering
the Brownian motion as an enhanced object by including the second
order structure given by an iterated integral process:
\[
\mathbf{X}_{s,t}=\big(X_{t}-X_{s},\int_{s<u_{1}<u_{2}<t}\mathrm{d}X_{u_{1}}\otimes\mathrm{d}X_{u_{2}}\big).
\]
In fact, given any function $(s,t)\rightarrow\mathbf{X}_{s,t}$ satisfying
certain algebraic and analytic conditions, a unique solution $Y$
to the equation (\ref{eq:SDE}) can be constructed in terms of $\mathbf{X}$,
so that the mapping $\mathbf{X}\rightarrow Y$ is continuous. Such
functions $\mathbf{X}$ are known as \textit{weakly geometric rough
paths}.

Lyons defined the solution for (\ref{eq:SDE}) effectively as 
\[
\mathbf{Y}_{s,t}=\big(Y_{t}-Y_{s},\int_{s<u_{1}<u_{2}<t}\mathrm{d}Y_{u_{1}}\otimes\mathrm{d}Y_{u_{2}}\big)
\]
so that the solution path $\mathbf{Y}$, like $\mathbf{X}$, is also
weakly geometric rough path. Lyons' rough path theory has an analytic
nature and goes way beyond the framework of Brownian motion. Later
on, Gubinelli \cite{Gub04} proposed an alternative way to interpret
the solution $Y$ as a \textit{controlled path}, which we will elaborate
below. The monograph of Friz and Hairer \cite{FH14} contains an excellent
exposition of the approach. Unlike the set of weakly geometric rough
paths, the set of controlled paths has a nice linear structure making
it a Banach space and some algebraic considerations are simplified
accordingly. Both \cite{Gub04,FH14} contains the complete theory
for the case when the Hölder exponenent $\beta$ of $\mathbf{X}$
is greater than $1/3$.

While for most parts it is commonly believed that the extension to
the case when $\beta<1/3$ is standard, the proofs and precise quantitative
estimates under the framework of controlled paths are not readily
available in the literature. Apart from this, there is an essential
ingredient whose extension to the case when $\beta<1/3$ is not obvious
at all. To be more specific, when formulating the differential equation
\begin{equation}
d{\cal Y}=F({\cal Y}){\rm d}{\bf X}\label{eq:RDEIntro}
\end{equation}
in the sense of controlled paths, one needs to prove that if $\mathcal{Y}$
is controlled by $\mathbf{X}$, then $F(\mathcal{Y})$ is also controlled
by ${\bf X}$ for suitably regular functions $F$. As we will see,
the challenge in this part has a strong algebraic nature that is not
similar to the usual H\"older regularity estimates. The ``geometric''
feature of ${\bf X}$ plays a critical role which is not needed in
the case when $\beta>1/3$. A major effort of the present article
is to develop this algebraic component carefully (cf. Section \ref{sec:Distortion-under-Lipschitz}
below). For completeness, we have also included a full proof towards
the well-posedness (existence, uniqueness and continuity) of the equation
(\ref{eq:RDEIntro}) under the framework of controlled paths. In our
modest opinion, having the controlled rough path framework properly
set-up in full generality along with the key quantitative estimates
may also be beneficial and convenient for the broader community.

Apart from Lyons' original approach and Gubinelli's controlled path
approach, there are numerous other approaches to study differential
equations driven by rough paths, some of which further develops the
idea of controlled paths (see for instance Davie \cite{Dav08}, Gubinelli
\cite{Gub10}, Hairer \cite{Hai14}, Lyons-Yang \cite{LY14}).

\vspace{2mm} \noindent \textbf{Organization.} The present article
is organized as follows. In Section 2, we recall the basic notions
of geometric rough paths and controlled rough paths. In Section 3,
we derive a H\"older estimate for controlled rough paths in terms
of the remainders. This estimate is needed for later purposes. In
Section 4, we prove the stability of controlled rough paths under
Lipschitz transformations. This part is a main ingredient of the present
article. In Sections 5 and 6, we study rough integration and rough
differential equations.

\section{Preliminary notions of rough paths}

We begin by recapturing some notions of geometric and controlled rough
paths over Banach spaces. This provides the basic framework on which
the present article is based.

\subsection{\label{subsec:GeoRP}Geometric rough paths}

Let $U$ and $V$ denote Banach spaces. The spaces $U$ and $V$ will
represent the space in which the paths $Y$ and $X$ in (\ref{eq:SDE})
take values respectively. A family of \textit{admissible tensor norms
}on $\big(V^{\otimes n}\big)_{n=1}^{\infty}$ (cf. Lyons-Qian \cite{LQ02})
is a family of norms, one for each of $V^{\otimes n}$, such that: 

\vspace{2mm} \noindent For $v\in V^{\otimes n}$ and $w\in V^{\otimes k}$,
\[
\Vert v\otimes w\Vert_{V^{\otimes(n+k)}}\leqslant\Vert v\Vert_{V^{\otimes n}}\Vert w\Vert_{V^{\otimes k}};
\]

\noindent Given a permutation $\sigma$ of order $n$, let $P_{\sigma}$
denote a linear transformation on $V^{\otimes n}$ such that 
\[
P_{\sigma}\big(v_{1}\otimes\ldots\otimes v_{n}\big)=v_{\sigma(1)}\otimes\ldots\otimes v_{\sigma(n)}.
\]
Then for all $v\in V^{\otimes n}$, 
\[
\Vert P_{\sigma}(v)\Vert_{V^{\otimes n}}=\Vert v\Vert_{V^{\otimes n}}.
\]
Throughout the rest, whenever working with Banach tensor products,
we always assume that a family of admissible tensor norms is given
fixed. For simplicity, we always use $|\cdot|$ to denote norms of
tensors, and use $\|\cdot\|$ to denote Hölder norms of paths.

Let ${\cal L}(U;V)$ denotes the space of bounded linear operators
from $U$ to $V$. We frequently identify spaces ${\cal L}(U;{\cal L}(U;V))$
and ${\cal L}(U^{\otimes2};V)$, and similarly for more general cases
$\mathcal{L}(U^{\otimes n};V)$.

Let $0<\beta\leqslant1/2$ and set $N\triangleq[1/\beta].$ The number
$\beta$ is fixed throughout this article, and all constants in the
article will, without further comment, depend on $\beta$.

A continuous mapping $X^{i}:\Delta_{T}\triangleq\{(s,t):0\leqslant s\leqslant t\leqslant T\}\rightarrow V^{\otimes i}$
is\textit{ $\beta$-Hölder continuous} if 
\[
\|X^{i}\|_{i\beta}\triangleq\sup_{0\leqslant s<t\leqslant T}\frac{|X_{s,t}^{i}|}{(t-s)^{i\beta}}<\infty.
\]
Let $T^{(N)}(V)$ denote the truncated tensor algebra $1\oplus V\oplus\cdots\oplus V^{\otimes N}$.
A mapping $\mathbf{X}:\Delta_{T}\rightarrow T^{(N)}(V)$ is called
\textit{multiplicative} if for any $s\leqslant u\leqslant t$,
\[
\mathbf{X}_{s,u}\otimes\mathbf{X}_{u,t}=\mathbf{X}_{s,t}.
\]

The following algebraic structure will be used in Section \ref{sec:Distortion-under-Lipschitz}
in a crucial way. For each $k\geqslant1$, consider the algebra 
\[
T^{(N)}(V)^{\boxtimes k}\triangleq\underbrace{T^{(N)}(V)\boxtimes\cdots\boxtimes T^{(N)}(V)}_{k}.
\]
Here $\boxtimes$ denotes the tensor product whose notation is used
to distinguish from the one $\otimes$ over $T^{(N)}(V).$ The product
structure $*$ over $T^{(N)}(V)^{\boxtimes k}$ is induced by 
\[
(\xi_{1}\boxtimes\cdots\boxtimes\xi_{k})*(\eta_{1}\boxtimes\cdots\boxtimes\eta_{k})\triangleq(\xi_{1}\otimes\eta_{1})\boxtimes\cdots\boxtimes(\xi_{k}\otimes\eta_{k}).
\]
If $f_{1},\cdots,f_{k}\in{\cal L}(T^{(N)}(V);U),$ we denote 
\[
f_{1}\boxtimes\cdots\boxtimes f_{k}:T^{(N)}(V)^{\boxtimes k}\rightarrow U^{\boxtimes k}
\]
as the mapping induced by 
\[
f_{1}\boxtimes\cdots\boxtimes f_{k}(\xi_{1}\boxtimes\cdots\boxtimes\xi_{k})\triangleq f_{1}(\xi_{1})\boxtimes\cdots\boxtimes f_{k}(\xi_{k}).
\]

There is an algebra homomorphism 
\[
\delta_{k}:(T^{(N)}(V),\otimes)\rightarrow\big(T^{(N)}(V)^{\boxtimes k},*\big)
\]
induced by 
\[
\delta_{k}(v)\triangleq v\boxtimes{\bf 1}\boxtimes\cdots\boxtimes{\bf 1}+\cdots+{\bf 1}\boxtimes\cdots\boxtimes{\bf 1}\boxtimes v,\ \ \ v\in V.
\]
See \cite{Reu93}, Section 1.4 for further details about $\delta_{k}$.
Let $\xi=v_{1}\otimes\cdots\otimes v_{r}\in V^{\otimes r}.$ Given
$I=\left\{ i_{1},\cdots,i_{m}\right\} $ with $i_{1}<\cdots<i_{m}$,
we define
\[
\xi|_{I}\triangleq v_{i_{1}}\otimes\ldots\otimes v_{i_{m}}
\]
and we adopt the convention that $\xi|_{\emptyset}$ is the scalar
$1$. One useful property of $\delta_{k}$ is that 

\begin{equation}
\delta_{k}\big(\xi\big)=\sum_{(I_{\alpha})}v|_{I_{1}}\boxtimes\ldots\boxtimes v|_{I_{k}}\label{eq:DeltaShuffle}
\end{equation}
where the above summation is taken over all partitions $(I_{\alpha})$
of $\{1,\ldots,r\}$ into disjoint subsets $I_{1},\cdots,I_{k}$ (some
of them can be $\emptyset$).

The \textit{free nilpotent group of order} $N$, denoted as $G^{(N)}(V)$,
is characterized by 
\begin{align}
G^{N}(V) & =\big\{\xi=(\xi^{0},\cdots,\xi^{N})\in T^{(N)}(V):\nonumber \\
 & \ \ \ \ \ \ \ \ \ \ \delta_{k}(\xi)=\sum_{0\leqslant l_{1}+\ldots+l_{k}\leqslant N}\xi^{l_{1}}\boxtimes\ldots\boxtimes\xi^{l_{k}}\;\forall k\geqslant2\big\}\label{eq:ShufPropk}
\end{align}

\begin{rem}
The above characterization of the free nilpotent group of order $N$
is equivalent to a common definition in terms of the exponential of
Lie series. Indeed, according to \cite{Reu93}, Theorem 3.2, an element
$\xi\in T^{(\infty)}(V)$ (the algebra of infinite tensor series)
is the exponential of a formal Lie series if and only if
\[
\tilde{\delta}_{2}(\xi)=\xi\boxtimes\xi,
\]
where $\tilde{\delta}_{2}$ is the canonical extension of $\delta_{2}$
onto $T^{(\infty)}(V)$. By a similar proof, this is also equivalent
to
\begin{equation}
\tilde{\delta}_{k}(\xi)=\xi^{\boxtimes k}\ \ \ \forall k\geqslant2.\label{eq:GroupLikeCharacterization}
\end{equation}
To see the equivalence between (\ref{eq:GroupLikeCharacterization})
and (\ref{eq:ShufPropk}), given $r$ and $k$ let us introduce the
projection 
\[
P_{r}:T^{(\infty)}(V)^{\boxtimes k}\rightarrow\bigoplus_{l_{1}+\ldots+l_{k}=r}V^{\otimes l_{1}}\boxtimes\cdots\boxtimes V^{\otimes l_{k}}\subseteq T^{(N)}(V)^{\boxtimes k}.
\]
If $\xi=(\xi^{0},\xi^{1},\xi^{2},\cdots)$ with $\xi^{i}\in V^{\otimes i}$,
then 
\[
P_{r}(\tilde{\delta}_{k}(\xi))=\sum_{l_{1}+\cdots+l_{k}=r}\xi^{l_{1}}\boxtimes\cdots\boxtimes\xi^{l_{k}}.
\]
As (\ref{eq:DeltaShuffle}) implies that $\delta_{k}$ sends $V^{\otimes r}$
to $\bigoplus_{l_{1}+\ldots+l_{k}=r}V^{\otimes l_{1}}\boxtimes\cdots\boxtimes V^{\otimes l_{k}}$,
we see that
\[
\delta_{k}(\xi^{r})=P_{r}(\tilde{\delta}_{k}(\xi))=\sum_{l_{1}+\cdots+l_{k}=r}\xi^{l_{1}}\boxtimes\cdots\boxtimes\xi^{l_{k}}.
\]
Therefore, if $(\xi^{0},\xi^{1},\ldots,\xi^{N})$ is the exponential
of a Lie series on $T^{(N)}(V)$, we have 
\[
\delta_{k}((\xi^{0},\ldots,\xi^{N}))=\sum_{0\leqslant l_{1}+\cdots+l_{k}\leqslant N}\xi^{l_{1}}\boxtimes\cdots\boxtimes\xi^{l_{k}}.
\]
\end{rem}
\begin{defn}
\label{def:geometricroughpath}A $\beta$-Hölder \textit{geometric
rough path} ${\bf X}$ is a multiplicative functional 
\[
{\bf X}=(1,X^{1},\cdots,X^{N}):\Delta_{T}\rightarrow T^{(N)}(V)
\]
such that $\mathbf{X}_{s,t}\in$$G^{(N)}(V)$ for any $(s,t)\in\Delta_{T}$
and $X^{i}$ is $\beta$-Hölder continuous for each $1\leqslant i\leqslant N$.
\end{defn}
\begin{rem}
Here we follow the convention in \cite{FH14} and call such rough
paths \textit{geometric}. In the earlier rough path literature (e.g.
\cite{FV10}), such paths are often called weakly geometric.
\end{rem}
Given two $\beta$-Hölder geometric rough paths ${\bf X},\tilde{{\bf X}},$
we define their ``distance'' by 
\begin{equation}
\rho_{\beta}({\bf X},\tilde{{\bf X}})\triangleq\sum_{i=1}^{N}\|X^{i}-\tilde{X}^{i}\|_{i\beta}.\label{eq:BetHMet}
\end{equation}
We also denote $\|{\bf X}\|_{\beta}\triangleq\rho_{\beta}({\bf X},{\bf 1})$.
\begin{rem}
A typical way of constructing geometric rough paths is as follows.
Let $\{X^{(m)}:m\geqslant1\}$ be sequence of continuous paths in
$V$ with bounded total variation. Then the limit

\begin{equation}
\mathbf{X}_{s,t}=\lim_{m\rightarrow\infty}\big(1,\int_{s<u_{1}<t}\mathrm{d}X_{u_{1}}^{(m)},\cdots,\int_{s<u_{1}<\cdots<u_{N}<t}\mathrm{d}X_{u_{1}}^{(m)}\otimes\cdots\otimes\mathrm{d}X_{u_{N}}^{(m)}\big)\label{eq:RoughPathConstruction}
\end{equation}
yields a $\beta$-Hölder geometric rough path provided that the convergence
holds under the $\beta$-Hölder metric (\ref{eq:BetHMet}). When $V$
is finite dimensional, the union over $\{\beta:\beta<\beta'\}$ of
all functionals $\Delta_{T}\rightarrow T^{(N)}(V)$ that can be constructed
through the procedure of (\ref{eq:RoughPathConstruction}) is precisely
the set of $\beta^{\prime}$-Hölder geometric rough paths (see \cite{FV10},
Corollary 8.24).

According to \cite{Lyo98}, when ${\bf X}$ is a geometric rough path,
the solution to the differential equation (\ref{eq:SDE}) can be constructed
in the sense of geometric rough paths. 
\end{rem}

\subsection{Controlled rough paths}

In this article, we take the perspective of controlled rough paths
introduced by Gubinelli \cite{Gub04}. A benefit is that the underlying
path space is a Banach space which simplifies the algebraic considerations
(for instance, when deriving continuity estimates) to some extent.
Heuristically, the solution to the rough differential equation $\mathrm{d}Y_{t}=F(Y_{t})\mathrm{d}X_{t}$
can formulated as the fixed point of the mapping $\mathcal{M}:Y_{\cdot}\rightarrow\int_{0}^{\cdot}F\left(Y_{t}\right)\mathrm{d}X_{t}$,
provided that ${\cal M}$ is a contraction on a suitable space of
controlled rough paths. We first define the notion of controlled rough
paths precisely.

Through out the rest of this article, we fix $\frac{1}{N+1}<\alpha<\beta\leqslant\frac{1}{N}\leqslant\frac{1}{2}$.
Let ${\bf X}$ be a given $\beta$-Hölder rough path over $V$. The
exponent $\alpha$ is used for the Hölder regularity of the controlled
rough paths to be introduced in what follows.
\begin{defn}
\label{def:CRP}A collection of continuous paths ${\cal Y}_{t}=(Y_{t}^{0},Y_{t}^{1},\cdots,Y_{t}^{N-1})$,
where $Y_{t}^{0}\in U$ and $Y_{t}^{i}\in{\cal L}(V^{\otimes i};U)$
for $1\leqslant i\leqslant N-1$, is called an ($\alpha$-\textit{Hölder})
\textit{controlled rough path} over $U$ with respect to ${\bf X},$
if the ``remainder'' defined by 
\[
{\cal RY}_{s,t}^{i}\triangleq\begin{cases}
Y_{t}^{i}-Y_{s}^{i}-\sum_{j=1}^{N-1-i}Y_{s}^{i+j}X_{s,t}^{j}, & \text{if }0\leqslant i\leqslant N-2,\\
Y_{t}^{N-1}-Y_{s}^{N-1}, & \text{if }i=N-1,
\end{cases}
\]
satisfies for each $0\leqslant i\leqslant N-1,$
\[
\|{\cal RY}^{i}\|_{(N-i)\alpha}\triangleq\sup_{0\leqslant s<t\leqslant T}\frac{|{\cal RY}_{s,t}^{i}|}{|t-s|^{(N-i)\alpha}}<\infty.
\]
The space of controlled rough paths over $U$ with respect to ${\bf X}$
is denoted as ${\cal D}_{{\bf X};\alpha}(U).$ We define a semi-norm
$\|\cdot\|_{{\bf X};\alpha}$ on ${\cal D}_{{\bf X};\alpha}(U)$ by
\[
\|{\cal Y}\|_{{\bf X};\alpha}\triangleq\sum_{i=0}^{N-1}\|{\cal RY}^{i}\|_{(N-i)\alpha}.
\]
\end{defn}
\begin{rem}
We often use the shorthanded notation $Y_{s,t}^{i}\triangleq Y_{t}^{i}-Y_{s}^{i}$.
\end{rem}
Let ${\bf X},\tilde{{\bf X}}$ be two $\beta$-Hölder rough paths.
To measure the distance between ${\cal Y}\in{\cal D}_{{\bf X};\alpha}(U)$
and $\tilde{{\cal Y}}\in{\cal D}_{\tilde{{\bf X}};\alpha}(U)$, we
define the functional 
\[
d_{{\bf X},\tilde{{\bf X}};\alpha}({\cal Y},\tilde{{\cal Y}})\triangleq\sum_{i=0}^{N-1}\|{\cal RY}^{i}-{\cal R}\tilde{{\cal Y}}^{i}\|_{(N-i)\alpha}.
\]

\vspace{2mm} \noindent \textbf{Notation}. In what follows, we always
use the notation $M(\cdots)$ to denote some universal function that
is continuous and increasing in every variable. Careful inspection
into the analysis shows that $M$ depends polynomially on every variable.

\section{Hölder estimates for controlled rough paths}

The following lemma tells us how to estimate $\|Y^{i}-\tilde{Y}^{i}\|_{\alpha}$
in terms of $d_{{\bf X},\tilde{{\bf X}};\alpha}({\cal Y},\tilde{{\cal Y}})$
and the difference of the initial data. This estimate is useful in
the next section when we study the stability of controlled paths under
Lipschitz transforms. For simplicity, we introduce the notation 
\[
\delta X^{i}\triangleq X^{i}-\tilde{X}^{i},\ \delta Y^{i}\triangleq Y^{i}-\tilde{Y}^{i},\ \delta{\cal R}^{i}\triangleq{\cal RY}^{i}-{\cal R\tilde{Y}}^{i}.
\]

\begin{lem}
\textcolor{blue}{\label{lem: estimating delta Y_alpha}}For each $2\leqslant i\leqslant N,$
there exists a continuous function $M_{i}:[0,\infty)^{5}\rightarrow[0,\infty)$,
increasing in each variable, such that 
\begin{align*}
\|\delta Y^{N-i}\|_{\alpha} & \leqslant M_{i}\big(T,\Vert\mathbf{X}\Vert_{\alpha},\Vert\tilde{\mathbf{X}}\Vert_{\alpha},\max_{1\leqslant j\leqslant i-1}|Y_{0}^{N-j}|,\max_{1\leqslant j\leqslant i-1}\|{\cal R}\mathcal{Y}^{N-j}\|_{j\alpha}\big)\\
 & \quad\;\times\big[\rho_{\alpha}({\bf X},\tilde{{\bf X}})+\|\delta{\cal R}^{N-i}\|_{i\alpha}+\sum_{j=1}^{i-1}\big(|\delta Y_{0}^{N-j}|+\|\delta{\cal R}^{N-j}\|_{j\alpha}\big)\big].
\end{align*}
\end{lem}
\begin{rem}
If $i=1$, $\delta Y^{N-i}=\delta\mathcal{R}^{N-1}$ and hence 
\begin{equation}
\Vert\delta Y^{N-i}\Vert_{\alpha}=\Vert\delta\mathcal{R}^{N-1}\Vert_{\alpha}.\label{eq:TopDegreeRemainder}
\end{equation}
\end{rem}
\begin{proof}
We prove the lemma by induction. When $i=2,$ we have 
\begin{align*}
Y_{s,t}^{N-2}-\tilde{Y}_{s,t}^{N-2} & =\big(Y_{s}^{N-1}X_{s,t}^{1}-\tilde{Y}_{s}^{N-1}\tilde{X}_{s,t}^{1}\big)+\big({\cal RY}_{s,t}^{N-2}-{\cal R}\tilde{{\cal Y}}_{s,t}^{N-2}\big)\\
 & =Y_{s}^{N-1}(X_{s,t}^{1}-\tilde{X}_{s,t}^{1})+(Y_{s}^{N-1}-\tilde{Y}_{s}^{N-1})\tilde{X}_{s,t}^{1}+\big({\cal RY}_{s,t}^{N-2}-{\cal R}\tilde{{\cal Y}}_{s,t}^{N-2}\big).
\end{align*}
It follows from (\ref{eq:TopDegreeRemainder}) that
\begin{align*}
\|\delta Y^{N-2}\|_{\alpha} & \leqslant\left(1+T^{\alpha}\right)\big(|Y_{0}^{N-1}|+\|Y^{N-1}\|_{\alpha}\big)\|\delta X^{1}\|_{\alpha}\\
 & \ \ \ +\left(1+T^{\alpha}\right)\big[\big(|\delta Y_{0}^{N-1}|+\|\delta Y^{N-1}\|_{\alpha}\big)\|\tilde{X}^{1}\|_{\alpha}+\|\delta{\cal R}^{N-2}\|_{2\alpha}T^{\alpha}\big]\\
 & \leqslant\left(1+T^{\alpha}\right)\big(T^{\alpha}+\Vert\tilde{{\bf X}}\Vert_{\alpha}+|Y_{0}^{N-1}|+\|{\cal R}\mathcal{Y}^{N-1}\|_{\alpha}\big)\\
 & \ \ \ \times\big[\|\delta X^{1}\|_{\alpha}+|\delta Y_{0}^{N-1}|+\|\delta{\cal R}^{N-1}\|_{\alpha}+\|\delta{\cal R}^{N-2}\|_{2\alpha}\big].
\end{align*}
Therefore, the claim holds in this case. 

Suppose that the claim holds for $\delta Y^{N-1},\cdots,\delta Y^{N-i}.$
Using that 
\begin{align*}
\delta Y_{s,t}^{N-(i+1)} & =\sum_{j=1}^{i}Y_{s}^{N-j}X_{s,t}^{i+1-j}-\sum_{j=1}^{i}\tilde{Y}_{s}^{N-j}\tilde{X}_{s,t}^{i+1-j}+\delta{\cal R}^{N-(i+1)},
\end{align*}
we have 
\begin{align}
 & \|\delta Y^{N-(i+1)}\|_{\alpha}\nonumber \\
 & \leqslant\big(1+T^{(i+1)\alpha}\big)\big[\sum_{j=1}^{i}\big(|Y_{0}^{N-j}|+\|Y^{N-j}\|_{\alpha}\big)\|\delta X^{i+1-j}\|_{(i+1-j)\alpha}\nonumber \\
 & \ \ \ +\sum_{j=1}^{i}\big(|\delta Y_{0}^{N-j}|+\|\delta Y^{N-j}\|_{\alpha}\big)\|\tilde{X}^{i+j}\|_{(i+j)\alpha}+\|\delta{\cal R}^{N-(i+1)}\|_{(i+1)\alpha}\big]\nonumber \\
 & \leqslant\big(1+T^{(i+1)\alpha}\big)\big(1+\max_{1\leqslant j\leqslant i}\big(|Y_{0}^{N-j}|+\|Y^{N-j}\|_{\alpha}\big)+\|\tilde{\mathbf{X}}\|_{\alpha}\big)\nonumber \\
 & \quad\;\big[\rho_{\alpha}({\bf X},\tilde{{\bf X}})+\sum_{j=1}^{i}\big(|\delta Y_{0}^{N-j}|+\|\delta Y^{N-j}\|_{\alpha}\big)+\|\delta{\cal R}^{N-(i+1)}\|_{(i+1)\alpha}\big].\label{eq:IndStep}
\end{align}
By the induction hypothesis with $\mathbf{X}=\tilde{\mathbf{X}}$
and taking $\tilde{Y}=0$, there is a continuous function $M_{j}$,
increasing in every variable, such that 
\begin{equation}
\|Y^{N-j}\|_{\alpha}\leqslant M_{j}\big(T,\Vert\mathbf{X}\Vert_{\alpha},\max_{1\leqslant l\leqslant j-1}|Y_{0}^{N-l}|,\max_{1\leqslant l\leqslant j}\|{\cal R}\mathcal{Y}^{N-l}\|_{\alpha}\big)\label{eq:YBound}
\end{equation}
and similarly for each $1\leqslant j\leqslant i$ we have 
\begin{align}
\|\delta Y^{N-j}\|_{\alpha} & \leqslant\tilde{M}_{j}\big(T,\Vert\mathbf{X}\Vert_{\alpha},\Vert\tilde{\mathbf{X}}\Vert_{\alpha},\max_{1\leqslant l\leqslant j-1}\big|Y_{0}^{N-l}\big|,\max_{1\leqslant l\leqslant j-1}\|{\cal R}\mathcal{Y}^{N-l}\|_{\alpha}\big)\nonumber \\
 & \quad\;\times\big[\rho_{\alpha}({\bf X},\tilde{{\bf X}})+\|\delta{\cal R}^{N-j}\|_{j\alpha}+\sum_{l=1}^{j-1}\big(\big|\delta Y_{0}^{N-l}\big|+\|\delta{\cal R}^{N-l}\|_{l\alpha}\big)\big].\label{eq:DelYBound}
\end{align}
The induction step follows by substituting (\ref{eq:YBound}) and
(\ref{eq:DelYBound}) into (\ref{eq:IndStep}).
\end{proof}
\begin{rem}
\label{rem:CompletenessOfMetric}An immediate consequence of Lemma
\ref{lem: estimating delta Y_alpha} is that ${\cal D}_{{\bf X};\alpha}(U)$
is a Banach space under the norm 
\begin{equation}
\interleave{\cal Y}\interleave_{{\bf X};\alpha}\triangleq\|{\cal Y}\|_{{\bf X};\alpha}+\sum_{i=0}^{N-1}|Y_{0}^{i}|.\label{eq:ControlNorm}
\end{equation}
\end{rem}

\section{\label{sec:Distortion-under-Lipschitz}Stability of controlled rough
paths under Lipschitz transforms}

Under the framework of controlled rough paths, an essential ingredient
for solving an RDE $d{\cal Y}=F({\cal Y}){\rm d}{\bf X}$ (with Lipschitz
vector field $F$) is to show that $F({\cal Y})$ is also a controlled
rough path. We would like to point out that the extension of this
property from the case of $1/3<\alpha\leqslant1/2$ (which is the
common setting in most of the literature) to the general case of $\alpha<1/3$
is non-trivial. As we will see, the main challenge here has an \textit{algebraic}
nature rather than just being the standard regularity estimates. To
point this out concisely, the Taylor expansion of $F$ for the $0$-th
level function (i.e. equation (\ref{eq:LipFctExp}) below when $j=0$)
allows us to motivate the full construction of $F({\cal Y})$ as a
controlled rough path in one go. However, checking the remainder regularity
conditions for all the derivative paths requires deeper algebraic
considerations and the geometric nature of ${\bf X}$ plays an essential
role. For this purpose, we take the viewpoint of Reutenauer \cite{Reu93}
and rely on the coproduct structure $\delta_{k}$ introduced in Section
\ref{subsec:GeoRP} in a crucial way.

We begin by recalling the notion of Lipschitz functions in the sense
to Stein \cite{Ste70}. Let $\mathcal{L}_{\mathrm{sym}}(V^{\boxtimes j};W)$
denote the set of linear bounded operators $T$ from $V^{\boxtimes j}$
to $W$ such that for all permutations $\sigma$ over $\{1,\ldots,j\}$,
\[
T(v_{\sigma(1)}\boxtimes\cdots\boxtimes v_{\sigma(n)})=T(v_{1}\boxtimes\cdots\boxtimes v_{n}).
\]

\begin{defn}
\label{def:LipFunc}Let $W,U$ be two Banach spaces and let $K$ be
a closed subset of $W$. Suppose that $\gamma\in(N,N+1]$ where $N$
is a non-negative integer. A collection of functions $F=(F^{0},F^{1},\cdots,F^{N})$
is said to be\textit{ $\gamma$-Lipschitz} over $K$, if:

\vspace{2mm}\noindent (i) the functions $F^{0}:K\rightarrow U$ and
$F^{j}:K\rightarrow{\cal L}_{{\rm sym}}(W^{\boxtimes j};U)$ ($1\leqslant j\leqslant N$)
are bounded on $K$;\\
(ii) for each $0\leqslant j\leqslant N$, the following Taylor expansion
holds:
\begin{equation}
F^{j}(y)(\xi)=\sum_{l=0}^{N-j}\frac{1}{l!}F^{j+l}(x)((y-x)^{\boxtimes l}\boxtimes\xi)+R_{j}(x,y)(\xi),\ \ x,y\in K,\xi\in W^{\boxtimes j},\label{eq:LipFctExp}
\end{equation}
where the remainder $R_{j}:K\times K\rightarrow{\cal L}_{{\rm sym}}(V^{\boxtimes j};W)$
satisfies 
\[
\sup_{x\neq y\in K}\frac{|R_{j}(x,y)|}{|x-y|^{\gamma-j}}<\infty\qquad\qquad\text{for all }0\leqslant j\leqslant N.
\]
The\textit{ ${\rm Lip}$-$\gamma$ norm} of $F$, denoted as $\|F\|_{\text{Lip-}\gamma}$,
is defined to be the smallest number $M>0$ such that for all $x,y\in K$,
\[
|F^{j}(x)|\leqslant M,\ |R_{j}(x,y)|\leqslant M|x-y|^{\gamma-j}
\]
for all $0\leqslant j\leqslant N$. The Banach space of all $\gamma$-Lipschitz
functions $F=(F^{0},\cdots,F^{N})$ is denoted as ${\rm Lip}(\gamma,K)$.
\end{defn}
Now let $\beta\in(0,1)$, $N\triangleq[1/\beta]$ and $\alpha\in(\frac{1}{N+1},\beta)$.
Let ${\bf X}$ be a given $\beta$-Hölder geometric rough path over
a Banach space $V$. Our aim in this section is to show that, if ${\cal Y}$
is an $\alpha$-Hölder controlled rough path over $W$ with respect
to ${\bf X}$, and $F=(F^{0},\cdots,F^{N})$ is $\gamma$-Lipschitz
over $W$ taking values in $U$, then $F({\cal Y})$ is an $\alpha$-Hölder
controlled rough path over $U$. In addition, given another controlled
rough path $\tilde{\mathcal{Y}}$, we shall establish a quantitative
continuity estimate of $d_{\mathbf{X},\tilde{\mathbf{X}};\alpha}(F(\mathcal{Y}),F(\tilde{\mathcal{Y})})$
in terms of $d_{\mathbf{X},\tilde{\mathbf{X}};\alpha}(\mathcal{Y},\tilde{\mathcal{Y}})$.

In the first place, we need to elaborate the meaning of $F({\cal Y})$
as a controlled rough path, which consists of the actual path in $U$
along with its $N-1$ derivative paths. The actual path, denoted as
$Z_{t}^{0},$ should apparently be given by $Z_{t}^{0}\triangleq F^{0}(Y_{t}^{0}).$
To motivate the derivative paths, we use the Taylor expansion of $F^{0}:$
\[
F^{0}(Y_{t}^{0})-F^{0}(Y_{s}^{0})\dot{=}\sum_{j=1}^{N-1}\frac{1}{j!}F^{j}(Y_{s}^{0})\big((Y_{s,t}^{0})^{\boxtimes j}\big),
\]
where $\dot{=}$ means being equal up to a term of regularity $|t-s|^{N\alpha}$.
Note that a term of such regularity is regarded as a remainder in
the expansion of $Z^{0}$. To proceed further, we adopt the convention
that $Y_{t}^{i}\in{\cal L}(V^{\otimes i};W)$ is extended to a linear
mapping from $T^{(N)}(V)$ to $T^{(N)}(W)$ by setting $Y_{t}^{i}(\xi)\triangleq0$
if $\xi\in V^{\otimes j}$ with $j\neq i$. Using the expansion of
$Y^{0},$ we have 
\begin{align*}
(Y_{s,t}^{0})^{\boxtimes j} & \dot{=}\big(\big(\sum_{i=1}^{N-1}Y_{s}^{i}\big)\mathbf{X}_{s,t}\big)^{\boxtimes j}=\big(\sum_{i=1}^{N-1}Y_{s}^{i}\big)^{\boxtimes j}\big(\mathbf{X}_{s,t}^{\boxtimes j}\big).
\end{align*}

Since $\mathbf{X}$ is a geometric rough path, $\mathbf{X}_{s,t}$
takes values in the free nilpotent group $G^{(N)}(V)$. By using (\ref{eq:ShufPropk}),
it is not hard to see that

\[
\delta_{j}(\mathbf{X}_{s,t})\dot{=}\mathbf{X}_{s,t}\boxtimes\cdots\boxtimes\mathbf{X}_{s,t}.
\]
As a result, we have 
\begin{align*}
F^{0}(Y_{t}^{0}) & \stackrel{\cdot}{=}\sum_{j=0}^{N-1}\frac{1}{j!}F^{j}(Y_{s}^{0})\big(\big(\sum_{i=1}^{N-1}Y_{s}^{i}\big)^{\boxtimes j}\big(\delta_{j}\big(\mathbf{X}_{s,t}\big)\big)\big)\\
 & =F^{0}(Y_{s}^{0})+\sum_{r=1}^{N-1}\sum_{j=1}^{N-1}\frac{1}{j!}F^{j}(Y_{s}^{0})\big(\sum_{i_{1}+\cdots+i_{j}=r}\big(Y_{s}^{i_{1}}\boxtimes\cdots\boxtimes Y_{s}^{i_{j}}\big)\big(\delta_{j}(X_{s,t}^{r})\big)\big),
\end{align*}
where the summation $\sum_{i_{1}+\cdots+i_{j}=r}$ is taken over all
$1\leqslant i_{1},\cdots,i_{j}\leqslant N-1$. It is then clear that
the derivative paths $Z^{1},\cdots,Z^{N-1}$ should be defined by
\begin{equation}
Z_{s}^{r}\triangleq\sum_{j=1}^{N-1}\frac{1}{j!}F^{j}(Y_{s}^{0})\big(\sum_{i_{1}+\cdots+i_{j}=r}(Y_{s}^{i_{1}}\boxtimes\cdots\boxtimes Y_{s}^{i_{j}})\circ\delta_{j}|_{V^{\otimes j}}\big).\label{eq:GDerZ}
\end{equation}
Note that the requirement $i_{1}+\cdots+i_{j}=r$ together with $i_{1},\cdots,i_{j}\geqslant1$
mean that the sum $\sum_{j=1}^{N-1}$ is in reality a sum $\sum_{j=1}^{r}$
as the terms from $j=r+1$ to $j=N-1$ are zero. We have left it as
$\sum_{j=1}^{N-1}$ for the convenience of interchanging summations
later on. Note that $Z_{s}^{r}\in{\cal L}(V^{\otimes r};U).$

To prove that $\mathcal{Z}=\left(Z^{0},\cdots,Z^{N-1}\right)$ is
controlled by $\mathbf{X}$, by Definition \ref{def:CRP} we need
to show that 
\begin{equation}
Z_{s,t}^{r}\dot{=}Z_{s}^{r+1}X_{s,t}^{1}+\cdots+Z_{s}^{N-1}X_{s,t}^{N-1-r}\label{eq:ConPropZ}
\end{equation}
for each $1\leqslant r\leqslant N-1$, where in this case $\dot{=}$
means being equal up to a term of regularity $|t-s|^{(N-r)\alpha}.$
The main challenge (and essence) of proving (\ref{eq:ConPropZ}) is
algebraic rather than analytic. In particular, this relies on a key
algebraic lemma which we now motivate.

First of all, there is nothing to prove when $r=0,$ since the definition
of $Z^{r}$ guarantees the desired regularity property in this case.
For $1\leqslant r\leqslant N-1,$ let $\xi\in V^{\otimes r}$ be a
generic element. To simplify the notation in the computation below,
we set 
\begin{equation}
\eta_{t}^{j}\triangleq\sum_{i_{1}+\cdots+i_{j}=r}(Y_{t}^{i_{1}}\boxtimes\cdots\boxtimes Y_{t}^{i_{j}})\circ\delta_{j}\left(\xi\right)\in W^{\boxtimes j}.\label{eq:TempTensor}
\end{equation}
Then we can write
\begin{align}
Z_{t}^{r}(\xi) & =\sum_{j=1}^{N-1}\frac{1}{j!}F^{j}(Y_{t}^{0})\big(\eta_{t}^{j}\big)\nonumber \\
 & \dot{=}\sum_{j=1}^{N-1}\frac{1}{j!}\big(\sum_{l=0}^{N-1-j}\frac{1}{l!}F^{j+l}(Y_{s}^{0})\big((Y_{s,t}^{0})^{\boxtimes l}\boxtimes\eta_{t}^{j}\big)\big)\label{eq:1=00003D}\\
 & =\sum_{j=1}^{N-1}\frac{1}{j!}\sum_{k=j}^{N-1}\frac{1}{(k-j)!}F^{k}(Y_{s}^{0})\big(\big(Y_{s,t}^{0}\big)^{\boxtimes(k-j)}\boxtimes\eta_{t}^{j}\big)\nonumber \\
 & =\sum_{k=1}^{N-1}\sum_{j=1}^{k}\frac{1}{j!(k-j)!}F^{k}(Y_{s}^{0})\big(\big(Y_{s,t}^{0}\big)^{\boxtimes(k-j)}\boxtimes\eta_{t}^{j}\big).\nonumber 
\end{align}
Let us define 
\begin{align}
\hat{\eta}_{t}^{j} & =\sum_{i_{1}+\cdots+i_{j}=r}\Big(\sum_{l_{1}\geqslant i_{1},\cdots,l_{j}\geqslant i_{j}}Y_{s}^{l_{1}}X_{s,t}^{l_{1}-i_{1}}\boxtimes\cdots\boxtimes Y_{s}^{l_{j}}X_{s,t}^{l_{j}-i_{j}}\Big)\circ\delta_{j}\left(\xi\right)\label{eq:ApproxTempTensor}
\end{align}
and 
\begin{equation}
\hat{Y}_{s,t}^{0}=\sum_{m=1}^{N-1}Y_{s}^{m}X_{s,t}^{m},\label{eq:ApproxTempTensor2}
\end{equation}
respectively. It follows that 
\begin{align}
Z_{t}^{r}(\xi) & \dot{=}\sum_{k=1}^{N-1}\sum_{j=1}^{k}\frac{1}{j!(k-j)!}F^{k}(Y_{s}^{0})\big(\big(Y_{s,t}^{0}\big)^{\boxtimes(k-j)}\boxtimes\eta_{t}^{j}\big)\nonumber \\
 & \dot{=}\sum_{k=1}^{N-1}\sum_{j=1}^{k}\frac{1}{j!(k-j)!}F^{k}(Y_{s}^{0})\big(\big(\hat{Y}_{s,t}^{0}\big)^{\boxtimes(k-j)}\boxtimes\hat{\eta}_{t}^{j}\big).\label{eq:2=00003D}
\end{align}
On the other hand, we have
\begin{align*}
 & (Z_{s}^{r}+Z_{s}^{r+1}X_{s,t}^{1}+\cdots+Z_{s}^{N-1}X_{s,t}^{N-1-r})(\xi)\\
 & =\sum_{l=r}^{N-1}\big(\sum_{k=1}^{N-1}\sum_{i_{1}+\cdots+i_{k}=l}\frac{1}{k!}F^{k}(Y_{s}^{0})\big(Y_{s}^{i_{1}}\boxtimes\cdots\boxtimes Y_{s}^{i_{k}}\big)\circ\delta_{k}\big(\mathbf{X}_{s,t}\otimes\xi\big)\big).
\end{align*}
Consequently, to prove $\mathcal{Z}$ is a controlled path, it boils
down to showing that 
\begin{align}
 & \sum_{k=1}^{N-1}\sum_{j=1}^{k}\frac{1}{j!(k-j)!}F^{k}(Y_{s}^{0})\big(\big(\hat{Y}_{s,t}^{0}\big)^{\boxtimes(k-j)}\boxtimes\hat{\eta}_{t}^{j}\big)\nonumber \\
 & \dot{=}\sum_{k=1}^{N-1}\sum_{l=1}^{N-1}\frac{F^{k}(Y_{s}^{0})}{k!}\big(\sum_{i_{1}+\cdots+i_{k}=l}\big(Y_{s}^{i_{1}}\boxtimes\cdots\boxtimes Y_{s}^{i_{k}}\big)\circ\delta_{k}\big(\mathbf{X}_{s,t}\otimes\xi\big)\big).\label{eq:AlgLem0}
\end{align}
Here an important point is that $F^{k}(Y_{s}^{0})$ is a symmetric
functional over $W^{\boxtimes k}$. To respect the underlying symmetry,
let $S_{k}:W^{\boxtimes k}\rightarrow W^{\boxtimes k}$ be the symmetrization
operator on homogeneous $k$-tensors, and let $K$ be its kernel.
We introduce the notation $\xi\stackrel{{\rm s}}{=}\eta$ to mean
that $\xi-\eta\in K$. Using the symmetry of $F^{k}(Y_{s}^{0})$,
it remains to establish the following algebraic lemma.
\begin{lem}
\label{lem:AlgLem}For each $1\leqslant k\leqslant N-1$, we have
\begin{align}
 & \sum_{j=1}^{k}\frac{1}{j!(k-j)!}\left(\hat{Y}_{s,t}^{0}\right)^{\boxtimes(k-j)}\boxtimes\hat{\eta}_{t}^{j}\nonumber \\
\stackrel{\mathrm{s}}{=} & \frac{1}{k!}\sum_{l=r}^{N-1}\sum_{i_{1}+\cdots+i_{k}=l}\left(Y_{s}^{i_{1}}\boxtimes\cdots\boxtimes Y_{s}^{i_{k}}\right)\circ\delta_{k}\left(\mathbf{X}_{s,t}\otimes\xi\right)+\Delta_{3;s,t}^{k},\label{eq:AlgLem}
\end{align}
where we recall that $\hat{\eta}_{t}^{j}$ is defined in (\ref{eq:ApproxTempTensor}),
$\hat{Y}_{s,t}^{0}$ is defined in (\ref{eq:ApproxTempTensor2}) and
\begin{equation}
\Delta_{3;s,t}^{k}\triangleq\frac{1}{k!}\sum_{\substack{1\leqslant i_{1},\cdots,i_{k}\leqslant N-1\\
i_{1}+\ldots+i_{k}\geqslant N
}
}Y_{s}^{i_{1}}\boxtimes\ldots\boxtimes Y_{s}^{i_{k}}\big(\big(\mathbf{X}_{s,t}\boxtimes\ldots\boxtimes\mathbf{X}_{s,t}\big)*\delta_{k}\big(\xi\big)\big).\label{eq:Delta3}
\end{equation}
\end{lem}
\begin{rem}
The role of the term $\Delta_{3;s,t}^{k}$ is to compensate the difference
between $\delta_{k}({\bf X}_{s,t})$ and ${\bf X}_{s,t}^{\boxtimes k}$
(cf. (\ref{eq:ShufPropk})), which arises from tensor truncation.
\end{rem}
\begin{proof}
By the linearity fir both sides of (\ref{eq:AlgLem}), it is enough
to consider the case $\xi=v_{1}\otimes\ldots\otimes v_{r}$ when $v_{i}\in V$
for all $1\leqslant i\leqslant r$. Recall that
\begin{equation}
\hat{\eta}_{t}^{j}=\sum_{i_{1}+\cdots+i_{j}=r}\Big(\sum_{l_{1}\geqslant i_{1},\ldots,l_{j}\geqslant i_{j}}Y_{s}^{l_{1}}X_{s,t}^{l_{1}-i_{1}}\boxtimes\cdots\boxtimes Y_{s}^{l_{j}}X_{s,t}^{l_{j}-i_{j}}\Big)\circ\delta_{j}\left(v_{1}\otimes\cdots\otimes v_{r}\right),\label{eq:AlgLine1}
\end{equation}
and 
\begin{equation}
\delta_{j}\big(v_{1}\otimes\cdots\otimes v_{r}\big)=\sum_{(I_{\alpha})}\xi|_{I_{1}}\boxtimes\cdots\boxtimes\xi|_{I_{j}},\label{eq:CoproductShuffle}
\end{equation}
where the summation is taken over all disjoint subsets $I_{1},\ldots,I_{j}$
such that $\cup_{\alpha=1}^{j}I_{\alpha}=\{1,\cdots,r\}$. Using the
above formula for $\delta_{j}$, equation (\ref{eq:AlgLine1}) becomes
\begin{align}
 & \sum_{i_{1}+\cdots+i_{j}=r}\big(\sum_{l_{1}\geqslant i_{1},\ldots,l_{j}\geqslant i_{j}}Y_{s}^{l_{1}}X_{s,t}^{l_{1}-i_{1}}\boxtimes\cdots\boxtimes Y_{s}^{l_{j}}X_{s,t}^{l_{j}-i_{j}}\big)\sum_{(I_{\alpha})}\xi|_{I_{1}}\boxtimes\cdots\boxtimes\xi|_{I_{j}}\nonumber \\
 & =\sum_{i_{1}+\cdots+i_{j}=r}\sum_{(I_{\alpha})}\sum_{l_{1}\geqslant i_{1},\ldots,l_{j}\geqslant i_{j}}Y_{s}^{l_{1}}\big(X_{s,t}^{l_{1}-i_{1}}\otimes\xi|_{I_{1}}\big)\boxtimes\cdots\boxtimes Y_{s}^{l_{j}}\big(X_{s,t}^{l_{j}-i_{j}}\otimes\xi|_{I_{j}}\big).\label{eq:AlgLine2}
\end{align}

Since $Y_{s}^{l}$ acts on $V^{\otimes l}$ and sends on all other
elements to zero, we know that 
\[
Y_{s}^{l}\big(X_{s,t}^{l-i}\otimes\xi|_{I}\big)=0\qquad\text{if }\left|I\right|\neq i.
\]
Therefore, the summation $\sum_{(I_{\alpha})}$ in (\ref{eq:AlgLine2})
becomes a summation over all partitions $(I_{\alpha})_{\alpha=1}^{j}$
of $\{1,\ldots r\}$ such that $\left|I_{\alpha}\right|=i_{\alpha}$
for all $\alpha$. As a result, we can write 
\[
\sum_{i_{1}+\cdots+i_{j}=r}\sum_{(I_{\alpha}){}_{\alpha=1}^{j}:|I_{\alpha}|=i_{\alpha}\,\forall\alpha}=\sum_{(I_{\alpha}){}_{\alpha=1}^{j}:|I_{\alpha}|\geqslant1\ \forall\alpha},
\]
where the right hand side denotes the summation over all partitions
$(I_{\alpha})_{\alpha=1}^{j}$ of $\{1,\cdots,r\}$ such that $\left|I_{\alpha}\right|\geqslant1$
for each $\alpha$. Moreover, as $Y_{s}^{l}(X_{s,t}^{q}\otimes\xi|_{I})=0$
unless $q=l-|I|$, we have 
\[
Y_{s}^{l}\big(X_{s,t}^{q}\otimes\xi|_{I}\big)=Y_{s}^{l}\big(\mathbf{X}_{s,t}\otimes\xi|_{I}\big).
\]
Note finally that as $\mathbf{X}_{s,t}\otimes\xi|_{I}$ has degree
at least $\left|I\right|$,
\[
Y_{s}^{l}\big(\mathbf{X}_{s,t}\otimes\xi|_{I}\big)=0\qquad\text{if }l<\left|I\right|.
\]
Therefore, for each $1\leqslant\alpha\leqslant j$ the summation $\sum_{l_{\alpha}\geqslant|I_{\alpha}|}$
can be replaced by the unrestricted sum $\sum_{l_{\alpha}=1}^{N-1}$.

Taking into account the above considerations, equation (\ref{eq:AlgLine2})
now becomes 
\[
\hat{\eta}_{t}^{j}=\sum_{(I_{\alpha}){}_{\alpha=1}^{j}:I_{\alpha}\neq\emptyset\,\forall\alpha}\sum_{l_{1},\cdots,l_{j}=1}^{N-1}Y_{s}^{l_{1}}\big(\mathbf{X}_{s,t}\otimes\xi|_{I_{1}}\big)\boxtimes\cdots\boxtimes Y_{s}^{l_{j}}\big(\mathbf{X}_{s,t}\otimes\xi|_{I_{j}}\big).
\]
It follows that 
\begin{align}
 & \sum_{j=1}^{k}\frac{1}{j!(k-j)!}\big(\hat{Y}_{s,t}^{0}\big)^{\boxtimes(k-j)}\boxtimes\hat{\eta}_{t}^{j}\nonumber \\
 & =\sum_{j=1}^{k}\frac{1}{j!(k-j)!}\sum_{m_{1},\cdots,m_{k-j}=1}^{N-1}Y_{s}^{m_{1}}\mathbf{X}_{s,t}\boxtimes\cdots\boxtimes Y_{s}^{m_{k-j}}\mathbf{X}_{s,t}\nonumber \\
 & \ \ \ \boxtimes\sum_{(I_{\alpha}){}_{\alpha=1}^{j}:I_{\alpha}\neq\emptyset\,\forall\alpha}\sum_{l_{1},\cdots,l_{j}=1}^{N-1}Y_{s}^{l_{1}}\big(\mathbf{X}_{s,t}\otimes\xi|_{I_{1}}\big)\boxtimes\cdots\boxtimes Y_{s}^{l_{j}}\big(\mathbf{X}_{s,t}\otimes\xi|_{I_{j}}\big)\nonumber \\
 & =\sum_{j=1}^{k}\frac{1}{j!(k-j)!}\sum_{h_{1},\cdots,h_{k}=1}^{N-1}Y_{s}^{h_{1}}\mathbf{X}_{s,t}\boxtimes\cdots\boxtimes Y_{s}^{h_{k-j}}\mathbf{X}_{s,t}\nonumber \\
 & \ \ \ \boxtimes\sum_{(I_{\alpha}){}_{\alpha=1}^{j}:I_{\alpha}\neq\emptyset\,\forall\alpha}Y_{s}^{h_{k-j+1}}\big(\mathbf{X}_{s,t}\otimes\xi|_{I_{1}}\big)\boxtimes\cdots\boxtimes Y_{s}^{h_{k}}\big(\mathbf{X}_{s,t}\otimes\xi|_{I_{j}}\big).\label{eq:AlgLine4}
\end{align}

As the next observation, let $\left(H_{i}\right)_{i=1}^{k}$ be a
partition of $\{1,\cdots,r\}$. If there exist $\beta_{1}<\cdots<\beta_{j}$
such that $H_{\beta_{i}}=I_{i}$ for $1\leqslant i\leqslant j$ and
$H_{i}=\emptyset$ for $i\notin\{\beta_{1},\cdots,\beta_{r}\}$, then
\begin{align*}
 & \sum_{h_{1},\cdots,h_{k}=1}^{N-1}Y_{s}^{h_{1}}\mathbf{X}_{s,t}\boxtimes\cdots\boxtimes Y_{s}^{h_{k-j}}\mathbf{X}_{s,t}\boxtimes Y_{s}^{h_{k-j+1}}\big(\mathbf{X}_{s,t}\otimes\xi|_{I_{1}}\big)\boxtimes\cdots\boxtimes Y_{s}^{h_{k}}\big(\mathbf{X}_{s,t}\otimes\xi|_{I_{j}}\big)\\
 & \ \ \ \stackrel{\mathrm{s}}{=}\sum_{h_{1},\cdots,h_{k}=1}^{N-1}Y_{s}^{h_{1}}\big(\mathbf{X}_{s,t}\otimes\xi|_{H_{1}}\big)\boxtimes\cdots\boxtimes Y_{s}^{h_{k}}\big(\mathbf{X}_{s,t}\otimes\xi|_{H_{k}}\big).
\end{align*}
There are a total of ${k \choose j}$ such partitions $(H_{i})_{i=1}^{k}$
for each given $j$-tuple $(I_{1},\cdots,I_{j})$. As a result, we
have
\begin{align*}
 & \sum_{h_{1},\cdots,h_{k}=1}^{N-1}\sum_{(H_{i}):\exists\beta_{1}<\cdots<\beta_{j},H_{\beta_{i}}=I_{i}\,\forall i}Y_{s}^{h_{1}}\big(\mathbf{X}_{s,t}\otimes\xi|_{H_{1}}\big)\boxtimes\cdots\boxtimes Y_{s}^{h_{k}}\Big(\mathbf{X}_{s,t}\otimes\xi|_{H_{k}}\big)\\
 & \ \ \ \stackrel{\mathrm{s}}{=}\sum_{h_{1},\cdots,h_{k}=1}^{N-1}{k \choose j}Y_{s}^{h_{1}}\mathbf{X}_{s,t}\boxtimes\cdots\boxtimes Y_{s}^{h_{k-j}}\mathbf{X}_{s,t}\boxtimes Y_{s}^{h_{k-j+1}}\big(\mathbf{X}_{s,t}\otimes\xi|_{I_{1}})\\
 & \ \ \ \ \ \ \boxtimes\cdots\boxtimes Y_{s}^{h_{k}}\big(\mathbf{X}_{s,t}\otimes\xi|_{I_{j}}\big).
\end{align*}
Since the summation
\[
\sum_{j=1}^{k}\sum_{(I_{\alpha}){}_{\alpha=1}^{j}:I_{\alpha}\neq\emptyset\,\forall\alpha}\sum_{(H_{i}):\exists\beta_{1}<\cdots<\beta_{j},H_{\beta_{i}}=I_{i}\,\forall i}
\]
is equivalent to summing over all partitions $\left(H_{i}\right)_{j=1}^{k}$
of $\{1,\cdots,r\},$ the expression in (\ref{eq:AlgLine4}) becomes
(up to permutation symmetry with respect to $\boxtimes$) 
\begin{align}
 & \frac{1}{k!}\sum_{h_{1},\cdots,h_{k}=1}^{N-1}\sum_{(H_{i})_{i=1}^{k}:\text{partition of }\{1,\cdots,r\}}Y_{s}^{h_{1}}\big(\mathbf{X}_{s,t}\otimes\xi|_{H_{1}}\big)\boxtimes\cdots\boxtimes Y_{s}^{h_{k}}\big(\mathbf{X}_{s,t}\otimes\xi|_{H_{k}}\big)\nonumber \\
 & =\frac{1}{k!}\sum_{h_{1},\cdots,h_{k}=1}^{N-1}Y_{s}^{h_{1}}\boxtimes\cdots\boxtimes Y_{s}^{h_{k}}\big(\big(\mathbf{X}_{s,t}\boxtimes\cdots\boxtimes\mathbf{X}_{s,t}\big)*\big(\sum_{(H_{i})_{i=1}^{k}}\xi|_{H_{1}}\boxtimes\cdots\boxtimes\xi|_{H_{k}}\big)\big).\label{eq:AlgLine5}
\end{align}
By using the formula (\ref{eq:CoproductShuffle}) for $\delta_{k}$
again, the expression in (\ref{eq:AlgLine5}) becomes 
\begin{align}
 & \frac{1}{k!}\sum_{h_{1},\cdots,h_{k}=1}^{N-1}Y_{s}^{h_{1}}\boxtimes\cdots\boxtimes Y_{s}^{h_{k}}\big(\big(\mathbf{X}_{s,t}\boxtimes\cdots\boxtimes\mathbf{X}_{s,t}\big)*\delta_{k}\big(v_{1}\otimes\cdots\otimes v_{r}\big)\big)\nonumber \\
 & =\frac{1}{k!}\sum_{l=r}^{N-1}\sum_{h_{1}+\cdots+h_{k}=l}^{N-1}Y_{s}^{h_{1}}\boxtimes\ldots\boxtimes Y_{s}^{h_{k}}\big(\big(\mathbf{X}_{s,t}\boxtimes\cdots\boxtimes\mathbf{X}_{s,t}\big)*\delta_{k}\big(v_{1}\otimes\cdots\otimes v_{r}\big)\big)+\Delta_{3;s,t}^{k},\label{eq:WithDelta3}
\end{align}
where $\Delta_{3;s,t}^{k}$ is defined to be the difference of the
two expressions in (\ref{eq:WithDelta3}).

Note that when $h_{1}+\cdots+h_{k}\leqslant N$, the operator $Y_{s}^{h_{1}}\boxtimes\cdots\boxtimes Y_{s}^{h_{k}}$
only acts non-trivially on elements of $\oplus_{l_{1}+\cdots+l_{k}\leqslant N}V^{\otimes l_{1}}\boxtimes\cdots\boxtimes V^{\otimes l_{k}}$.
Since ${\bf X}_{s,t}\in G^{(N)}(V)$, according to the shuffle product
formula (\ref{eq:ShufPropk}), for such $h_{i}$'s we have 
\begin{align*}
 & Y_{s}^{h_{1}}\boxtimes\cdots\boxtimes Y_{s}^{h_{k}}\big(\big(\mathbf{X}_{s,t}\boxtimes\cdots\boxtimes\mathbf{X}_{s,t}\big)*\delta_{k}\big(v_{1}\otimes\cdots\otimes v_{r}\big)\big)\\
 & =Y_{s}^{h_{1}}\boxtimes\cdots\boxtimes Y_{s}^{h_{k}}\big(\sum_{l_{1}+\cdots+l_{k}\leqslant N}\big(X_{s,t}^{l_{1}}\boxtimes\cdots\boxtimes X_{s,t}^{l_{k}}\big)*\delta_{k}\big(v_{1}\otimes\cdots\otimes v_{r}\big)\big)\\
 & =Y_{s}^{h_{1}}\boxtimes\cdots\boxtimes Y_{s}^{h_{k}}\big(\delta_{k}\big(\mathbf{X}_{s,t}\big)*\delta_{k}\big(v_{1}\otimes\cdots\otimes v_{r}\big)\big).
\end{align*}
Since $\delta_{k}$ is a $*$-homomorphism, the expression in (\ref{eq:WithDelta3})
becomes
\begin{align*}
\frac{1}{k!} & \sum_{l=r}^{N-1}\sum_{h_{1}+\ldots+h_{k}=l}^{N-1}Y_{s}^{h_{1}}\boxtimes\cdots\boxtimes Y_{s}^{h_{k}}\big(\delta_{k}\big(\mathbf{X}_{s,t}\otimes v_{1}\otimes\cdots\otimes v_{r}\big)\big)+\Delta_{3;s,t}^{k},
\end{align*}
which is precisely the right hand side of (\ref{eq:AlgLem}).
\end{proof}
Having the above algebraic considerations, we can now prove the main
result of this section. For the need in the study of RDEs in the next
section, we also establish a continuity estimate for Lipschitz transformations.
\begin{thm}
\label{lem:LipschitzLemma} (i) {[}Stability{]} Let ${\cal Y}$ be
a controlled rough path over $[0,T]$ with respect to ${\bf X}.$
Let $F=(F^{0},F^{1},\cdots,F^{N})$ be a $\gamma$-Lipschitz function
with $\gamma\in(N,N+1]$. Then the path $\mathcal{Z}=F(\text{\ensuremath{\mathcal{Y}}})$
as defined in (\ref{eq:GDerZ}) is a path controlled by $\mathbf{X}$
in the sense of Definition \ref{def:CRP}. In addition, we have 
\begin{equation}
\Vert F\left(\mathcal{Y}\right)\Vert_{\mathbf{X};\alpha}\leqslant\Vert F\Vert_{\text{Lip-}N}\cdot M\big(T,\max_{1\leqslant i\leqslant N-1}|Y_{0}^{i}|,\big\|\mathcal{Y}\big\|{}_{\mathbf{X};\alpha},\Vert\mathbf{X}\Vert_{\alpha}\big)\label{eq:SingleLipschitz}
\end{equation}
(ii) {[}Continuity estimate{]} Let $\mathcal{Y}$ and $\tilde{\mathcal{Y}}$
be paths over $\left[0,T\right]$ controlled by $\mathbf{X}$ and
$\tilde{\mathbf{X}}$ respectively, and let $F$ be $(N+1)$-Lipschitz.
Then we have

\begin{align}
 & d_{\mathbf{X},\tilde{\mathbf{X}};\alpha}\big(F\big(\mathcal{Y}\big),F\big(\tilde{\mathcal{Y}}\big)\big)\nonumber \\
 & \leqslant\Vert F\Vert_{\text{Lip-}(N+1)}\cdot M\big(T,\max_{1\leqslant i\leqslant N-1}|Y_{0}^{i}|,\max_{1\leqslant i\leqslant N-1}|\tilde{Y}_{0}^{i}|,\big\|\mathcal{Y}\big\|{}_{\mathbf{X};\alpha},\big\|\tilde{\mathcal{Y}}\big\|{}_{\tilde{\mathbf{X}};\alpha},\Vert\mathbf{X}\Vert_{\alpha},\Vert\tilde{\mathbf{X}}\Vert_{\alpha}\big)\nonumber \\
 & \ \ \ \times\big(\max_{0\leqslant i\leqslant N-1}|Y_{0}^{i}-\tilde{Y}_{0}^{i}|+d_{\mathbf{X},\tilde{\mathbf{X}};\alpha}\big(\mathcal{Y},\tilde{\mathcal{Y}}\big)+\rho_{\alpha}({\bf X},\tilde{{\bf X}})\big).\label{eq:LipschitzDifference}
\end{align}
In both parts, $M(\cdots)$ denotes a continuous function that is
increasing in every variable.
\end{thm}
\begin{proof}
To ease our discussion, we use the notation ``$\lesssim$'' to denote
an estimate up to a continuous increasing function $M$ in $T$, $\max_{1\leqslant i\leqslant N-1}|Y_{0}^{i}|$,
$\max_{1\leqslant i\leqslant N-1}|\tilde{Y}_{0}^{i}|$, $||\mathcal{Y}||_{{\bf X};\alpha}$,
$||\tilde{\mathcal{Y}}||_{\tilde{{\bf X}};\alpha}$, $\Vert\mathbf{X}\Vert_{\alpha}$,
$\Vert\tilde{\mathbf{X}}\Vert_{\alpha}$, which may differ from line
to line. We closely follow the notation used earlier in the algebraic
considerations. In particular, in order to prove the theorem, essentially
we need to keep track of the remainder from each of the notions ``$\dot{=}$''
appearing earlier.

First of all, as seen before, we can write 
\begin{align*}
Z_{t}^{r}(\xi) & =\sum_{k=0}^{N-1}F^{k}(Y_{s}^{0})(C_{k}+\Delta_{2;s,t}^{k})+\Delta_{1;s,t},\\
EZ_{s,t}^{r} & (\xi)=\sum_{k=0}^{N-1}F^{k}(Y_{s}^{0})(D_{k}-\Delta_{3;s,t}^{k}),
\end{align*}
Here 
\[
EZ_{s,t}^{r}(\xi)\triangleq(Z_{s}^{r}+Z_{s}^{r+1}X_{s,t}^{1}+\cdots+Z_{s}^{N-1}X_{s,t}^{N-1-r})(\xi),
\]
$C_{k}$, $D_{k}$ are defined by the left and right hand sides of
the algebraic identity (\ref{eq:AlgLem}) respectively. The remainders
$\Delta_{1;s,t},\Delta_{2;s,t}^{k}$ are associated with the notions
``$\dot{=}$'' appearing earlier and $\Delta_{3;,s,t}^{k}$ is defined
by (\ref{eq:Delta3}). To be precise, they are defined by the following
equations.

\vspace{2mm} \noindent (i) (cf. (\ref{eq:1=00003D}) and Taylor's
theorem with integral form remainder) 
\begin{align}
\Delta_{1;s,t} & \triangleq\sum_{j=1}^{r}\int_{0}^{1}\frac{(1-\theta)^{N-1-j}}{j!(N-1-j)!}F^{N}(Y_{s}^{0}+\theta Y_{s,t}^{0})\big((Y_{s,t}^{0})^{\boxtimes(N-j)}\boxtimes\eta_{t}^{j}\big){\rm d}\theta,\label{eq:Delta1}
\end{align}
where we recall that $\sum_{j=1}^{r}=\sum_{j=1}^{N-1}.$\\
(ii) (cf. (\ref{eq:2=00003D}))
\begin{align*}
\Delta_{2;s,t}^{k} & \triangleq\sum_{j=1}^{r}\frac{1}{j!(k-j)!}\sum_{i_{1}+\cdots+i_{j}=r}\big((Y_{s,t}^{0})^{\boxtimes(k-j)}\boxtimes\big(Y_{t}^{i_{1}}\boxtimes\cdots\boxtimes Y_{t}^{i_{j}}(\delta_{j}(\xi))\big)\\
 & \ \ \ -\big(Y_{s}^{1}X_{s,t}^{1}+\cdots+Y_{s}^{N-1}X_{s,t}^{N-1}\big)^{\boxtimes(k-j)}\boxtimes\big(EY_{s,t}^{i_{1}}\boxtimes\cdots\boxtimes EY_{s,t}^{i_{j}}(\delta_{j}(\xi))\big)\big),
\end{align*}
where 
\[
EY_{s,t}^{i}\triangleq Y_{s}^{i}+Y_{s}^{i+1}X_{s,t}^{1}+\cdots+Y_{s}^{N-1}X_{s,t}^{N-1-i}.
\]
According to Lemma \ref{lem:AlgLem}, we have
\[
F^{k}(Y_{s}^{0})(C_{k})=F^{k}(Y_{s}^{0})(D_{k}).
\]
It follows that 
\begin{equation}
{\cal RZ}_{s,t}^{r}=\sum_{k=1}^{N-1}F^{k}(Y_{s}^{0})\left(\Delta_{2;s,t}^{k}+\Delta_{3;s,t}^{k}\right)+\Delta_{1;s,t}.\label{eq:RemainderLip}
\end{equation}
Similar definitions and identities hold for the tilde-quantities.
We need to estimate the regularity of the $\Delta$'s.

As a standard way, we frequently use the simple inequality 
\begin{equation}
|ab-\tilde{a}\tilde{b}|\leqslant|a-\tilde{a}|\cdot|b|+|\tilde{a}|\cdot|b-\tilde{b}|.\label{eq:trivialmethod}
\end{equation}
Also note that $|F^{k}(Y_{s}^{0})|\lesssim\|F\|_{\text{Lip-}N}$.
As a consequence of Lemma \ref{lem: estimating delta Y_alpha} (without
the presence of $\tilde{{\cal Y}}$), we have 
\begin{equation}
|Y_{s,t}^{0}|\lesssim\left|t-s\right|^{\alpha},\quad|Y_{t}^{i}|\lesssim1\;\forall i\geqslant1.\label{eq:BoundGubinelli}
\end{equation}
From these considerations and the expression (\ref{eq:Delta1}) of
$\Delta_{1};s,t$, we see that 
\begin{equation}
|\Delta_{1;s,t}|\lesssim\|F\|_{\text{Lip-}N}\cdot|t-s|^{(N-r)\alpha}.\label{eq:FirstLipRemainder}
\end{equation}

For the term $\Delta_{2;s,t}^{k}$, by forming a telescoping sum it
boils down to estimating 
\begin{align}
 & \big((Y_{s,t}^{0})^{\boxtimes(k-j)}\big)\boxtimes\big(Y_{t}^{i_{1}}\boxtimes\cdots\boxtimes Y_{t}^{i_{j}}(\delta_{j}(\xi))-\big(Y_{s}^{1}X_{s,t}^{1}+\cdots+Y_{s}^{N-1}X_{s,t}^{N-1}\big)^{\boxtimes(k-j)}\nonumber \\
 & \ \ \ \boxtimes\big(EY_{s,t}^{i_{1}}\boxtimes\cdots\boxtimes EY_{s,t}^{i_{j}}(\delta_{j}(\xi))\big)\big)\nonumber \\
 & =(Y_{s,t}^{0}-Y_{s}^{1}X_{s,t}^{1}+\cdots+Y_{s}^{N-1}X_{s,t}^{N-1})\boxtimes(Y_{s,t}^{0})^{\boxtimes(k-j-1)}\boxtimes Y_{t}^{i_{1}}\cdots\boxtimes Y_{t}^{i_{j}}(\delta_{j}(\xi))+\cdots\nonumber \\
 & \ \ \ +\left(Y_{s}^{1}X_{s,t}^{1}+\cdots+Y_{s}^{N-1}X_{s,t}^{N-1}\right)^{\boxtimes(k-j)}\boxtimes\left(EY_{s,t}^{i_{1}}\boxtimes\cdots\boxtimes(Y_{t}^{i_{j}}-EY_{s,t}^{i_{j}})(\delta_{j}(\xi))\right).\label{eq:DeltaTwoStep}
\end{align}
Note that 
\[
\big|Y_{s,t}^{0}-Y_{s}^{1}X_{s,t}^{1}+\cdots+Y_{s}^{N-1}X_{s,t}^{N-1}\big|\lesssim\left|t-s\right|^{N\alpha}
\]
and 
\begin{align*}
\big|Y_{s}^{1}X_{s,t}^{1}+\cdots+Y_{s}^{N-1}X_{s,t}^{N-1}\big| & =\big|Y_{s,t}^{0}-\mathcal{R}\mathcal{Y}_{s,t}^{0}\big|\lesssim\left|t-s\right|^{\alpha}.
\end{align*}
In addition, for each $i\geqslant1$ we have
\begin{align*}
\big|EY_{s,t}^{i}\big| & =\big|Y_{s}^{i}+Y_{s}^{i+1}X_{s,t}^{1}+\cdots+Y_{s}^{N-1}X_{s,t}^{N-1-i}\big|\\
 & =\big|Y_{t}^{i}-\mathcal{R}\mathcal{Y}_{s,t}^{i}\big|\lesssim1,
\end{align*}
and for each $i\leqslant r$ we have
\[
\big|Y_{t}^{i}-EY_{s,t}^{i}\big|\lesssim\left|t-s\right|^{(N-r)\alpha}.
\]
Consequently, we see that

\begin{equation}
|\Delta_{2;s,t}^{k}|\lesssim|t-s|^{(N-r)\alpha}.\label{eq:SecondLipRemainder}
\end{equation}

We now estimate 
\begin{align*}
\Delta_{3;s,t}^{k} & =\frac{1}{k!}\sum_{h_{1}+\cdots+h_{k}\geqslant N}Y_{s}^{h_{1}}\boxtimes\cdots\boxtimes Y_{s}^{h_{k}}\big(\big(\mathbf{X}_{s,t}\boxtimes\cdots\boxtimes\mathbf{X}_{s,t}\big)*\delta_{k}\big(\xi\big)\big)\\
 & =\frac{1}{k!}\sum_{h_{1}+\cdots+h_{k}\geqslant N}Y_{s}^{h_{1}}\boxtimes\cdots\boxtimes Y_{s}^{h_{k}}\big(\big(\sum_{l_{1},\cdots,l_{k}=1}^{N}X_{s,t}^{l_{1}}\boxtimes\cdots\boxtimes X_{s,t}^{l_{k}}\big)*\delta_{k}\big(\xi\big)\big).
\end{align*}
Since $h_{1}+\ldots+h_{k}\geqslant N$ and $\xi\in V^{\otimes r}$,
the only non-zero terms are the ones when $l_{1}+\ldots+l_{k}\geqslant N-r$.
In this case, we see that 
\[
|X_{s,t}^{l_{1}}\boxtimes\cdots\boxtimes X_{s,t}^{l_{k}}|\lesssim(t-s)^{(N-r)\alpha}.
\]
Therefore, 
\begin{align*}
|\Delta_{3;s,t}^{k}| & \lesssim\sum_{h_{1}+\ldots+h_{k}\geqslant N}|Y_{s}^{h_{1}}|\cdots|Y_{s}^{h_{k}}|\cdot(t-s)^{(N-r)\alpha}\lesssim(t-s)^{(N-r)\alpha}.
\end{align*}

We particularly point out that the constant hidden within ``$\lesssim$''
is independent of $Y_{0}^{0}$, which will be important for RDE considerations
later on. From (\ref{eq:RemainderLip}), (\ref{eq:FirstLipRemainder})
and (\ref{eq:SecondLipRemainder}) we conclude that $\mathcal{Z}$
is controlled by $\mathbf{X}$ and the estimate (\ref{eq:SingleLipschitz})
follows.

To prove the second part the theorem, we need to estimate 
\begin{align*}
 & {\cal RZ}_{s,t}^{r}-{\cal R}\tilde{{\cal Z}}_{s,t}^{r}\\
 & =\sum_{k=1}^{N-1}\big(F^{k}(Y_{s}^{0})\big(\Delta_{2;s,t}^{k}\big)-F^{k}(\tilde{Y}_{s}^{0})\big(\tilde{\Delta}_{2;s,t}^{k}\big)\big)\\
 & \ \ \ +\sum_{k=1}^{N-1}\big(F^{k}(Y_{s}^{0})\big(\Delta_{3;s,t}^{k}\big)-F^{k}(\tilde{Y}_{s}^{0})\big(\tilde{\Delta}_{3;s,t}^{k}\big)\big)+\big(\Delta_{1;s,t}-\tilde{\Delta}_{1;s,t}\big)
\end{align*}
For this purpose, let us introduce 
\[
D({\bf X},{\cal Y};\tilde{{\bf X}},\tilde{{\cal Y}})\triangleq\rho_{\alpha}({\bf X},\tilde{{\bf X}})+d_{{\bf X},\tilde{{\bf X}};\alpha}({\cal Y},\tilde{{\cal Y}})+\sum_{i=0}^{N-1}|Y_{0}^{i}-\tilde{Y}_{0}^{i}|.
\]
Now it remains to establish the following set of estimates (for $1\leqslant k\leqslant N-1$):
\begin{align}
|F^{k}(Y_{s}^{0})-F^{k}(\tilde{Y}_{s}^{0})| & \lesssim\|F\|_{\text{Lip-}N}\cdot D({\bf X},{\cal Y};\tilde{{\bf X}},\tilde{{\cal Y}}),\nonumber \\
|\Delta_{1;s,t}-\tilde{\Delta}_{1;s,t}| & \lesssim\|F\|_{\text{Lip-}(N+1)}\cdot D({\bf X},{\cal Y};\tilde{{\bf X}},\tilde{{\cal Y}})\cdot|t-s|^{(N-r)\alpha},\label{eq:DeltaOneDifference}\\
|\Delta_{2;s,t}^{k}-\tilde{\Delta}_{2;s,t}^{k}| & \lesssim D({\bf X},{\cal Y};\tilde{{\bf X}},\tilde{{\cal Y}})\cdot|t-s|^{(N-r)\alpha},\label{eq:DeltaTwoDifference}\\
|\Delta_{3;s,t}^{k}-\tilde{\Delta}_{3;s,t}^{k}| & \lesssim D({\bf X},{\cal Y};\tilde{{\bf X}},\tilde{{\cal Y}})\cdot|t-s|^{(N-r)\alpha}.\label{eq:DeltaThreeDifference}
\end{align}

To see the first inequality, first note that
\begin{align*}
 & |F^{k}(Y_{s}^{0})-F^{k}(\tilde{Y}_{s}^{0})|\\
= & \big|\int_{0}^{1}F^{k+1}\big(Y_{s}^{0}+\theta\big(Y_{s}^{0}-\tilde{Y}_{s}^{0}\big)\big)\big(Y_{s}^{0}-\tilde{Y}_{s}^{0}\big)\mathrm{d}\theta\big|\\
\lesssim & \|F\|_{\text{Lip-}N}\cdot\big(\big|Y_{0}^{0}-\tilde{Y}_{0}^{0}\big|+\big\| Y^{0}-\tilde{Y}^{0}\big\|_{\alpha}\big).
\end{align*}
The inequality then follows from Lemma \ref{lem: estimating delta Y_alpha}.

For the inequality (\ref{eq:DeltaOneDifference}), according to its
expression (\ref{eq:Delta1}), it suffices to estimate 
\[
F^{N}(Y_{s}^{0}+\theta Y_{s,t}^{0})(\eta_{t}^{j}\boxtimes(Y_{s,t}^{0})^{\boxtimes(N-j)})-F^{N}(\tilde{Y}_{s}^{0}+\theta\tilde{Y}_{s,t}^{0})(\tilde{\eta}_{t}^{j}\boxtimes(\tilde{Y}_{s,t}^{0})^{\boxtimes(N-j)}),
\]
Recall from the definition (\ref{eq:TempTensor}) of $\eta_{t}^{j}$
that
\[
\eta_{t}^{j}-\tilde{\eta}_{t}^{j}=\sum_{i_{1}+\cdots+i_{j}=r}(Y_{t}^{i_{1}}\boxtimes\cdots\boxtimes Y_{t}^{i_{j}}-\tilde{Y}_{t}^{i_{1}}\boxtimes\cdots\boxtimes\tilde{Y}_{t}^{i_{j}})\circ\delta_{j}\left(\xi\right).
\]
For each $1\leqslant i\leqslant N-1$, we have 
\begin{align}
\big|Y_{t}^{i}-\tilde{Y}_{t}^{i}\big| & \leqslant\big|Y_{0}^{i}-\tilde{Y}_{0}^{i}\big|+\big|\big(Y_{t}^{i}-Y_{0}^{i}\big)-\big(\tilde{Y}_{t}^{i}-\tilde{Y}_{0}^{i}\big)\big|\nonumber \\
 & \leqslant\big|Y_{0}^{i}-\tilde{Y}_{0}^{i}\big|+\Vert Y^{i}-\tilde{Y}^{i}\Vert_{\alpha}T^{\alpha}\nonumber \\
 & \lesssim D({\bf X},{\cal Y};\tilde{{\bf X}},\tilde{{\cal Y}}).\label{eq:DifferenceGub}
\end{align}
Since $|Y_{t}^{i}|\lesssim1$, $|\tilde{Y}_{t}^{i}|\lesssim1$, it
follows that 
\[
\big|\eta_{t}^{j}-\tilde{\eta}_{t}^{j}\big|\lesssim D({\bf X},{\cal Y};\tilde{{\bf X}},\tilde{{\cal Y}}).
\]
On the other hand, according to Lemma \ref{lem: estimating delta Y_alpha}
we have
\begin{align*}
\big|Y_{s,t}^{0}-\tilde{Y}_{s,t}^{0}\big| & \leqslant\Vert Y^{0}-\tilde{Y}^{0}\Vert_{\alpha}\left(t-s\right)^{\alpha}\lesssim D({\bf X},{\cal Y};\tilde{{\bf X}},\tilde{{\cal Y}})\left(t-s\right)^{\alpha}.
\end{align*}
Therefore, we see that 
\begin{align*}
 & \big|F^{N}(Y_{s}^{0}+\theta Y_{s,t}^{0})-F^{N}(\tilde{Y}_{s}^{0}+\theta\tilde{Y}_{s,t}^{0})\big|\\
 & \leqslant\Vert F\Vert_{\text{Lip-}(N+1)}\big|Y_{s}^{0}+\theta Y_{s,t}^{0}-\big(\tilde{Y}_{s}^{0}+\theta\tilde{Y}_{s,t}^{0}\big)\big|\\
 & \leqslant\Vert F\Vert_{\text{Lip-}(N+1)}\big|Y_{s}^{0}+\theta Y_{s,t}^{0}-\big(\tilde{Y}_{s}^{0}+\theta\tilde{Y}_{s,t}^{0}\big)\big|\\
 & \leqslant\Vert F\Vert_{\text{Lip-}(N+1)}\big(\big|Y_{0}^{0}-\tilde{Y}_{0}^{0}\big|+\big|Y_{s,t}^{0}-\tilde{Y}_{s,t}^{0}\big|+\big|Y_{0,s}^{0}-\tilde{Y}_{0,s}^{0}\big|\big)\\
 & \leqslant\Vert F\Vert_{\text{Lip-}(N+1)}\big(\big|Y_{0}^{0}-\tilde{Y}_{0}^{0}\big|+2\Vert Y^{0}-\tilde{Y}^{0}\Vert_{\alpha}T^{\alpha}\big)\\
 & \lesssim\Vert F\Vert_{\text{Lip-}(N+1)}D({\bf X},{\cal Y};\tilde{{\bf X}},\tilde{{\cal Y}}).
\end{align*}
The inequality (\ref{eq:DeltaOneDifference}) thus follows (the regularity
$|t-s|^{(N-r)\alpha}$ comes from the fact that $j\leqslant r$ in
the summation (\ref{eq:Delta1})).

For the inequality (\ref{eq:DeltaTwoDifference}), to estimate $\Delta_{2;s,t}^{k}-\tilde{\Delta}_{2;s,t}^{k}$
we write this difference in the form of a telescoping sum that is
similar to (\ref{eq:DeltaTwoStep}) but also with the tilde-quantities.
We already have the required estimates for $Y_{s,t}^{0}-\tilde{Y}_{s,t}^{0}$
and $Y_{t}^{i}-\tilde{Y}_{t}^{i}$ when analyzing $\Delta_{1;s,t}$.
We also have 
\begin{align*}
 & \big|\big(Y_{s}^{1}X_{s,t}^{1}+\cdots+Y_{s}^{N-1}X_{s,t}^{N-1}\big)-\big(\tilde{Y}_{s}^{1}\tilde{X}_{s,t}^{1}+\cdots+\tilde{Y}_{s}^{N-1}\tilde{X}_{s,t}^{N-1}\big)\big|\\
 & \leqslant\big|Y_{s,t}^{0}-\tilde{Y}_{s,t}^{0}\big|+\big|\mathcal{R}\mathcal{Y}_{s,t}^{0}-\mathcal{R}\tilde{\mathcal{Y}}_{s,t}^{0}\big|\\
 & \lesssim D({\bf X},{\cal Y};\tilde{{\bf X}},\tilde{{\cal Y}})\big|t-s\big|^{\alpha},
\end{align*}
and 
\begin{align*}
\big|EY_{s,t}^{i}-E\tilde{Y}_{s,t}^{i}\big|\leqslant & \big|Y_{t}^{i}-EY_{s,t}^{i}-\big(\tilde{Y}_{t}^{i}-E\tilde{Y}_{s,t}^{i}\big)\big|+\big|Y_{t}^{i}-\tilde{Y}_{t}^{i}\big|\\
= & \big|\mathcal{R}\mathcal{Y}_{s,t}^{i}-\mathcal{R}\tilde{\mathcal{Y}}_{s,t}^{i}\big|+\big|Y_{t}^{i}-\tilde{Y}_{t}^{i}\big|\\
\lesssim & D({\bf X},{\cal Y};\tilde{{\bf X}},\tilde{{\cal Y}}).
\end{align*}
As a result, the desired estimate for $\Delta_{2;s,t}^{k}-\tilde{\Delta}_{2;s,t}^{k}$
follows.

For the last inequality (\ref{eq:DeltaThreeDifference}), we have
\begin{align*}
\\
\Delta_{3;s,t}^{k}-\tilde{\Delta}_{3;s,t}^{k} & =\frac{1}{k!}\sum_{h_{1}+\cdots+h_{k}\geqslant N}Y_{s}^{h_{1}}\boxtimes\cdots\boxtimes Y_{s}^{h_{k}}\big(\big(\sum_{l_{1},\cdots,l_{k}=1}^{N}X_{s,t}^{l_{1}}\boxtimes\cdots\boxtimes X_{s,t}^{l_{k}}\big)*\delta_{k}\big(\xi\big)\big)\\
 & \ \ \ -\tilde{Y}_{s}^{h_{1}}\boxtimes\cdots\boxtimes\tilde{Y}_{s}^{h_{k}}\big(\big(\sum_{l_{1},\cdots,l_{k}=1}^{N}\tilde{X}_{s,t}^{l_{1}}\boxtimes\cdots\boxtimes\tilde{X}_{s,t}^{l_{k}}\big)*\delta_{k}\big(\xi\big)\big).
\end{align*}
Note that $\big|Y_{t}^{i}\big|\lesssim1$, $\big|\tilde{Y}_{t}^{i}\big|\lesssim1$
and $\big|Y_{t}^{i}-\tilde{Y}_{t}^{i}\big|\lesssim D({\bf X},{\cal Y};\tilde{{\bf X}},\tilde{{\cal Y}})$.
We also have $|X_{s,t}^{l}|\lesssim(t-s)^{l\alpha}$, $|\tilde{X}_{s,t}^{l}|\lesssim(t-s)^{l\alpha}$
and $|X_{s,t}^{l}-\tilde{X}_{s,t}^{l}|\lesssim D({\bf X},{\cal Y};\tilde{{\bf X}},\tilde{{\cal Y}})(t-s)^{l\alpha}$.
Therefore, we obtain the desired inequality (\ref{eq:DeltaThreeDifference}).

Now the proof of the theorem is complete.
\end{proof}

\section{Continuity of rough integrals}

In this section, we study the integral $\int{\cal Z}d{\bf X}$ as
a controlled rough path and establish a continuity estimate. We fix
$\frac{1}{N+1}<\alpha<\beta\leqslant\frac{1}{N}\leqslant\frac{1}{2}.$
Given a partition $\mathcal{P}:s=t_{0}<t_{1}<\cdots<t_{n}=t$, we
set 
\[
|\mathcal{P}|\triangleq\max_{0\leqslant i\leqslant n-1}(t_{i+1}-t_{i}).
\]
All paths below are defined on $[0,T]$.
\begin{prop}
\label{lem:Integration-Lemma} (i) Let ${\bf X}$ be a $\beta$-Hölder
geometric rough path over $V,$ and let ${\cal Z}$ be an $\alpha$-Hölder
controlled rough path over ${\cal L}(V;U)$ with respect to ${\bf X}$.
Then the following limit exists: 
\begin{equation}
\int_{s}^{t}Z\mathrm{d}X\triangleq\lim_{|\mathcal{P}|\rightarrow0}\sum_{t_{i}\in\mathcal{P}}\sum_{k=1}^{N}Z_{t_{i}}^{k-1}X_{t_{i},t_{i+1}}^{k}.\label{eq:roughintegral}
\end{equation}
In addition, if we define $\int_{0}^{\cdot}{\cal Z}d{\bf X}\triangleq(I^{0},I^{1},\cdots,I^{N-1})$
by
\begin{equation}
I_{t}^{0}\triangleq\int_{0}^{t}Z{\rm d}X,\ I_{t}^{1}\triangleq Z_{t}^{0},\cdots,\ I_{t}^{N-1}\triangleq Z_{t}^{N-2},\label{eq:RoughIntegralControl}
\end{equation}
Then the path $\int_{0}^{\cdot}{\cal Z}\mathrm{d}X=\left(I^{0},I^{1},\ldots,I^{N-1}\right)$
is an $\alpha$-Hölder controlled rough path with respect to $\mathbf{X}$.\\
(ii) (Continuity estimates) Let $\mathbf{X}$ and $\tilde{\mathbf{X}}$
be $\beta$-Hölder geometric rough paths, and let ${\cal Z},\tilde{\mathcal{Z}}$
be paths controlled by $\mathbf{X}$ and $\tilde{\mathbf{X}}$ respectively.
We use $\mathcal{I}=\int_{0}^{\cdot}{\cal Z}\mathrm{d}{\bf X}$ and
$\tilde{\mathcal{I}}=\int_{0}^{\cdot}\tilde{{\cal Z}}\mathrm{d}\tilde{{\bf X}}$
to denote the controlled paths obtained by integrating $\mathcal{{\cal Z}}$
and $\tilde{\mathcal{{\cal Z}}}$ respectively. Then there exists
a function $M:[0,\infty)^{5}\rightarrow[0,\infty)$ that is continuous
and increasing in every variable, such that
\begin{align}
d_{\mathbf{X},\tilde{\mathbf{X}};\alpha}\big(\mathcal{I},\tilde{\mathcal{I}}\big)\leqslant & \max\big(T^{\alpha},T^{\beta-\alpha}\big)M\big(T,\big\|\mathbf{X}\big\|_{\beta},\Vert\tilde{\mathbf{X}}\Vert_{\beta},\big\|\tilde{\mathcal{Z}}\big\|_{\tilde{\mathbf{X}};\alpha},\big|\tilde{Z}_{0}^{N-1}\big|\big)\nonumber \\
 & \times\big(d_{\mathbf{X},\tilde{\mathbf{X}};\alpha}\big(\mathcal{Z},\tilde{\mathcal{Z}}\big)+\rho_{\beta}({\bf X},\tilde{{\bf X}})+\big|Z_{0}^{N-1}-\tilde{Z}_{0}^{N-1}\big|\big).\label{eq:ContEstRI}
\end{align}
\end{prop}
\begin{proof}
Let $s<t$ be fixed. Given any partition ${\cal P}$ of $[s,t]$,
we denote 
\[
\int_{\mathcal{P}}Z\mathrm{d}X\triangleq\sum_{t_{i}\in\mathcal{P}}\sum_{k=1}^{N}Z_{t_{i}}^{k-1}X_{t_{i},t_{i+1}}^{k}
\]
 Then 
\begin{align*}
 & \int_{\mathcal{P}}Z\mathrm{d}X-\int_{\mathcal{P}\backslash\{t_{j}\}}Z\mathrm{d}X\\
 & =\sum_{k=1}^{N}Z_{t_{j-1}}^{k-1}X_{t_{j-1},t_{j}}^{k}+\sum_{k=1}^{N}Z_{t_{j}}^{k-1}X_{t_{j},t_{j+1}}^{k}-\sum_{k=1}^{N}Z_{t_{j-1}}^{k-1}X_{t_{j-1},t_{j+1}}^{k}\\
 & =\sum_{k=1}^{N}Z_{t_{j-1}}^{k-1}X_{t_{j-1},t_{j}}^{k}+\sum_{k=1}^{N}Z_{t_{j}}^{k-1}X_{t_{j},t_{j+1}}^{k}-\sum_{k=1}^{N}\sum_{l=0}^{k}Z_{t_{j-1}}^{k-1}X_{t_{j-1},t_{j}}^{k-l}\otimes X_{t_{j},t_{j+1}}^{l}\\
 & =\sum_{k=1}^{N}Z_{t_{j}}^{k-1}X_{t_{j},t_{j+1}}^{k}-\sum_{k=1}^{N}\sum_{l=1}^{k}Z_{t_{j-1}}^{k-1}X_{t_{j-1},t_{j}}^{k-l}\otimes X_{t_{j},t_{j+1}}^{l}.
\end{align*}
We claim that the last expression is equal to $\sum_{k=1}^{N}\mathcal{R}\mathcal{Z}_{t_{j-1},t_{j}}^{k-1}\otimes X_{t_{j},t_{j+1}}^{k}$.
Indeed, by writing 
\[
Z_{t_{j}}^{k-1}=Z_{t_{j-1}}^{k-1}+\sum_{r=1}^{N-k}Z_{t_{j-1}}^{k+r-1}X_{t_{j-1},t_{j}}^{r}+{\cal R}{\cal Z}_{t_{j-1},t_{j}}^{k-1},
\]
it is equivalent to seeing that 
\begin{align*}
 & \sum_{k=1}^{N}Z_{t_{j-1}}^{k-1}X_{t_{j},t_{j+1}}^{k}+\sum_{k=1}^{N}\sum_{r=1}^{N-k}Z_{t_{j-1}}^{k+r-1}X_{t_{j-1},t_{j}}^{r}\otimes X_{t_{j},t_{j+1}}^{k}\\
 & \ \ \ -\sum_{k=1}^{N}\sum_{l=1}^{k}Z_{t_{j-1}}^{k-1}X_{t_{j-1},t_{j}}^{k-l}\otimes X_{t_{j},t_{j+1}}^{l}=0.
\end{align*}
The above equation follows by interchanging the order of summation
in the middle term. Consequently, we arrive at
\begin{equation}
\int_{\mathcal{P}}Z\mathrm{d}X-\int_{\mathcal{P}\backslash\{t_{j}\}}Z\mathrm{d}X=\sum_{k=1}^{N}\mathcal{R}\mathcal{Z}_{t_{j-1},t_{j}}^{k-1}\otimes X_{t_{j},t_{j+1}}^{k}.\label{eq:CtyRI}
\end{equation}

In the following argument, we directly consider the continuity estimate.
The case of the single ${\cal I}$ (without the tilde-paths) is easier
and only requires minor modification. Using (\ref{eq:CtyRI}), we
have 
\begin{align*}
 & \big|\int_{\mathcal{P}}Z\mathrm{d}X-\int_{\mathcal{P}\backslash\{t_{j}\}}Z\mathrm{d}X-\big(\int_{\mathcal{P}}\tilde{Z}\mathrm{d}\tilde{X}-\int_{\mathcal{P}\backslash\{t_{j}\}}\tilde{Z}\mathrm{d}\tilde{X}\big)\big|\\
 & =\big|\sum_{k=1}^{N}\mathcal{R}\mathcal{Z}_{t_{j-1},t_{j}}^{k-1}\otimes X_{t_{j},t_{j+1}}^{k}-\mathcal{R}\tilde{\mathcal{Z}}_{t_{j-1},t_{j}}^{k-1}\otimes\tilde{X}_{t_{j},t_{j+1}}^{k}\big|\\
 & \leqslant\big(d_{{\bf X},\tilde{{\bf X}};\alpha}({\cal Z},\tilde{{\cal Z}})\big\|\mathbf{X}\big\|_{\alpha}+\big\|\tilde{\mathcal{Z}}\big\|_{\tilde{\mathbf{X}};\alpha}\rho_{\alpha}({\bf X},\tilde{{\bf X}})\big)\cdot\big(t_{j+1}-t_{j-1}\big)^{\left(N+1\right)\alpha}.
\end{align*}
As $\sum_{j=1}^{n-1}\left(t_{j+1}-t_{j-1}\right)\leqslant2(t-s)$,
we may choose a $j$ such that 
\[
t_{j+1}-t_{j-1}\leqslant\frac{2\left(t-s\right)}{n-1}.
\]
It follows that

\begin{align*}
 & \big|\int_{\mathcal{P}}Z\mathrm{d}X-\int_{\mathcal{P}\backslash\{t_{j}\}}Z\mathrm{d}X-\big(\int_{\mathcal{P}}\tilde{Z}\mathrm{d}\tilde{X}-\int_{\mathcal{P}\backslash\{t_{j}\}}\tilde{Z}\mathrm{d}\tilde{X}\big)\big|\\
 & \leqslant\big(\frac{2}{n-1}\big)^{(N+1)\alpha}\left(t-s\right)^{(N+1)\alpha}\big(d_{{\bf X},\tilde{{\bf X}};\alpha}({\cal Z},\tilde{{\cal Z}})\left\Vert \mathbf{X}\right\Vert _{\alpha}+\big\|\tilde{\mathcal{Z}}\big\|_{\tilde{\mathbf{X}};\alpha}\rho_{\alpha}({\bf X},\tilde{{\bf X}})\big).
\end{align*}
By successively removing partition points from $\mathcal{P}$, we
arrive at
\begin{align}
 & \big|\int_{\mathcal{P}}Z\mathrm{d}X-\int_{\{s,t\}}Z\mathrm{d}X-\big(\int_{\mathcal{P}}\tilde{Z}\mathrm{d}\tilde{X}-\int_{\{s,t\}}\tilde{Z}\mathrm{d}\tilde{X}\big)\big|\nonumber \\
 & \leqslant C_{N,\alpha}\left(t-s\right)^{\left(N+1\right)\alpha}\big(d_{{\bf X},\tilde{{\bf X}};\alpha}({\cal Z},\tilde{{\cal Z}})\left\Vert \mathbf{X}\right\Vert _{\alpha}+\big\|\tilde{\mathcal{Z}}_{\tilde{\mathbf{X}};\alpha}\big\|\rho_{\alpha}({\bf X},\tilde{{\bf X}})\big),\label{eq:GeneralMaximalBound}
\end{align}
where 
\[
C_{N,\alpha}\triangleq\sum_{n=3}^{\infty}\big(\frac{2}{n-1}\big)^{(N+1)\alpha}.
\]
The version of the inequality (\ref{eq:GeneralMaximalBound}) without
the tilde-paths is easily seen to be 
\begin{equation}
\big|\int_{{\cal P}}Z{\rm d}X-\int_{\{s,t\}}Z{\rm d}X\big|\leqslant C_{N,\alpha}\|{\cal Z}\|_{{\bf X};\alpha}\|{\bf X}\|_{\alpha}(t-s)^{(N+1)\alpha}.\label{eq:GenMaxIneNOTild}
\end{equation}

We now use the inequality (\ref{eq:GenMaxIneNOTild}) to show that
the limit 
\[
\lim_{|\mathcal{P}|\rightarrow0}\int_{\mathcal{P}}Z\mathrm{d}X
\]
exists. Let $\hat{\mathcal{P}}$ and $\tilde{\mathcal{P}}$ be partitions
over $[s,t]$, and let $\hat{\mathcal{P}}\vee\tilde{\mathcal{P}}$
be the partition obtained by taking a union of the partition points
from $\hat{\mathcal{P}}$ and $\tilde{\mathcal{P}}$. For each pair
$(s_{l},s_{l+1})$ of adjacent points in $\hat{\mathcal{P}}$, by
applying the estimate (\ref{eq:GenMaxIneNOTild}) to the partition
$\hat{\mathcal{P}}\vee\tilde{\mathcal{P}}\cap[s_{l},s_{l+1}]$, we
obtain that 
\begin{align*}
\big|\int_{\hat{\mathcal{P}}\vee\tilde{\mathcal{P}}\cap[s_{l},s_{l+1}]}Z\mathrm{d}X-\int_{\{s_{l},s_{l+1}\}}Z\mathrm{d}X\big| & \leqslant C_{N,\alpha}\big(s_{l+1}-s_{l}\big)^{\left(N+1\right)\alpha}\left\Vert \mathcal{Z}\right\Vert _{\tilde{\mathbf{X}};\alpha}\left\Vert \mathbf{X}\right\Vert _{\alpha}.
\end{align*}
By summing over $l$, we have
\begin{align*}
 & \big|\int_{\hat{\mathcal{P}}\vee\tilde{\mathcal{P}}}Z\mathrm{d}X-\int_{\hat{\mathcal{P}}}Z\mathrm{d}X\big|\\
 & \leqslant\sum_{l}\big|\int_{\hat{\mathcal{P}}\vee\tilde{\mathcal{P}}\cap[s_{l},s_{l+1}]}Z\mathrm{d}X-\int_{\{s_{l},s_{l+1}\}}Z\mathrm{d}X\big|\\
 & \leqslant C_{N,\alpha}|\hat{\mathcal{P}}|^{\left(N+1\right)\alpha-1}\left\Vert \mathcal{Z}\right\Vert _{\tilde{\mathbf{X}};\alpha}\left\Vert \mathbf{X}\right\Vert _{\alpha}(t-s).
\end{align*}
A similar inequality holds with $\hat{\mathcal{P}}$ replaced by $\tilde{\mathcal{P}}$.
Using the triangle inequality, we end up with an estimate for $\int_{\hat{\mathcal{P}}}Z\mathrm{d}X-\int_{\tilde{\mathcal{P}}}Z\mathrm{d}X$,
from which we can deduce the convergence of (\ref{eq:roughintegral})
using the Cauchy criterion.

Next, we establish the continuity estimate (\ref{eq:ContEstRI}).
By taking $|{\cal P}|\rightarrow0$ in (\ref{eq:GeneralMaximalBound})
and using the definition of ${\cal I},\tilde{{\cal I}}$, we have
\begin{align*}
 & \big|\int_{s}^{t}Z\mathrm{d}X-\sum_{k=1}^{N}Z_{s}^{k-1}X_{s,t}^{k}-\int_{s}^{t}\tilde{Z}\mathrm{d}\tilde{X}-\sum_{k=1}^{N}\tilde{Z}_{s}^{k-1}\tilde{X}_{s,t}^{k}\big|\\
 & =\big|I_{t}^{0}-\sum_{k=1}^{N-1}I_{s}^{k}X_{s,t}^{k}-\big(\tilde{I}_{t}^{0}-\sum_{k=1}^{N-1}\tilde{I}_{s}^{k}\tilde{X}_{s,t}^{k}\big)-\big(Z_{s}^{N-1}X_{s,t}^{N}-\tilde{Z}_{s}^{N-1}\tilde{X}_{s,t}^{N}\big)\big|\\
 & \leqslant C_{N,\alpha}\left(t-s\right)^{\left(N+1\right)\alpha}\big(d_{{\bf X},\tilde{{\bf X}};\alpha}({\cal Z},\tilde{{\cal Z}})\left\Vert \mathbf{X}\right\Vert _{\alpha}+\big\|\tilde{\mathcal{Z}}\big\|_{\tilde{\mathbf{X}};\alpha}\rho_{\alpha}({\bf X},\tilde{{\bf X}})\big).
\end{align*}
Note that 
\begin{align*}
 & \big|Z_{s}^{N-1}-\tilde{Z}_{s}^{N-1}\big|\cdot\big|X_{s,t}^{N}\big|+\big|\tilde{Z}_{s}^{N-1}\big|\cdot\big|X_{s,t}^{N}-\tilde{X}_{s,t}^{N}\big|\\
 & \leqslant\big\|\mathcal{R}Z^{N-1}-\mathcal{R}\tilde{Z}^{N-1}\big\|_{\alpha}\left\Vert \mathbf{X}\right\Vert _{\beta}\left(t-s\right)^{N\beta}T^{\alpha}+\big|Z_{0}^{N-1}-\tilde{Z}_{0}^{N-1}\big|\left\Vert \mathbf{X}\right\Vert _{\beta}\left(t-s\right)^{N\beta}\\
 & \ \ \ +\big\|\tilde{\mathcal{Z}}\big\|_{\tilde{\mathbf{X}};\alpha}\rho_{\beta}({\bf X},\tilde{{\bf X}})\left(t-s\right)^{N\beta}T^{\alpha}+\big|\tilde{Z}_{0}^{N-1}\big|\rho_{\beta}({\bf X},\tilde{{\bf X}})\left(t-s\right)^{N\beta}.
\end{align*}
Therefore, we obtain that 
\begin{align*}
 & \big|I_{t}^{0}-\sum_{k=1}^{N-1}I_{s}^{k}X_{s,t}^{k}-\big(\tilde{I}_{t}^{0}-\sum_{k=1}^{N-1}\tilde{I}_{s}^{k}\tilde{X}_{s,t}^{k}\big)\big|\\
 & \leqslant2C_{N,\alpha}\left(t-s\right)^{N\alpha}T^{\alpha}\big(d_{{\bf X},\tilde{{\bf X}};\alpha}({\cal Z},\tilde{{\cal Z}})\left\Vert \mathbf{X}\right\Vert _{\alpha}+\big\|\tilde{\mathcal{Z}}\big\|_{\tilde{\mathbf{X}};\alpha}\rho_{\beta}({\bf X},\tilde{{\bf X}})\big)\\
 & \ \ \ +\big|Z_{0}^{N-1}-\tilde{Z}_{0}^{N-1}\big|\left\Vert \mathbf{X}\right\Vert _{\beta}\left(t-s\right)^{N\beta}+\big|\tilde{Z}_{0}^{N-1}\big|\rho_{\beta}({\bf X},\tilde{{\bf X}})\left(t-s\right)^{N\beta}.
\end{align*}
The desired estimate (\ref{eq:ContEstRI}) thus follows. Note that
from the above estimate (the analogue without the tilde-paths), it
is clear that ${\cal I}$ is a controlled rough path (namely the remainders
have the desired regularity properties).
\end{proof}

\section{Rough differential equations}

We now proceed to establish existence, uniqueness and continuity of
solutions for the RDE 
\begin{equation}
d{\cal Y}_{t}=F({\cal Y}_{t}){\rm d}{\bf X}_{t}\label{eq:RDE}
\end{equation}
in the space of controlled rough paths. As a standard idea, this is
formulated as a fixed point problem for the transformation 
\[
{\cal M}:{\cal Y}\mapsto Y_{0}+\int F({\cal Y}){\rm d}{\bf X}.
\]
We first derive a continuity estimate for ${\cal M}.$ Using such
continuity estimate, we then show that ${\cal M}$ is a contraction
on a small time interval. The general case follows from a patching
argument.

Throughout the rest, let $\frac{1}{N+1}\leqslant\alpha<\beta<\frac{1}{N}\leqslant\frac{1}{2}$
be fixed. Let $F=(F^{0},\cdots,F^{N})$ be a given Lip-$(N+1)$ function
defined on $U$ and taking values in ${\cal L}(V;U)$.

\subsection{Composition of Lipschitz transform and rough integration}

In this subsection, we consider paths defined on $[0,\tau]$ ($\tau>0$
is given fixed).
\begin{lem}
\label{lem:ContractionEstimates} Let $\mathbf{X}$ and $\tilde{\mathbf{X}}$
be $\beta$-Hölder geometric rough paths over $V.$ Let ${\cal Y}$
and $\tilde{{\cal Y}}$ be $U$-valued paths controlled by $\mathbf{X}$
and $\tilde{\mathbf{X}}$ respectively. Define the controlled rough
path $\mathcal{J}=(J^{0},\cdots,J^{N-1})$ with respect to $\mathbf{X}$
in the following way
\begin{equation}
J_{t}^{0}\triangleq Y_{0}^{0}+\big(\int_{0}^{\cdot}F({\cal Y}){\rm d}{\bf X}\big)_{t}^{0},\ J_{t}^{i}=\big(\int_{0}^{\cdot}F\left(\mathcal{Y}\right)\mathrm{d}{\bf X}\big)_{t}^{i}\quad\text{for }i\geqslant1.\label{eq:JPath}
\end{equation}
Define $\tilde{{\cal J}}$ controlled by $\tilde{{\bf X}}$ in a similar
way. Then the following estimates hold true:

\begin{equation}
\big\|{\cal J}\big\|_{{\bf X};\alpha}\leqslant\max\big(\tau^{\alpha},\tau^{\beta-\alpha}\big)M\big(\tau,\Vert\mathbf{X}\Vert_{\beta},\max_{1\leqslant i\leqslant N-1}\big|Y_{0}^{i}\big|,\big\|\mathcal{Y}\big\|_{{\bf X};\alpha},\Vert F\Vert_{\text{Lip-}N}\big)\label{eq:CompEst}
\end{equation}
and

\begin{align}
d_{\mathbf{X},\tilde{{\bf X}};\alpha}\big(\mathcal{J},\tilde{\mathcal{J}}\big) & \leqslant\max\left(\tau^{\alpha},\tau^{\beta-\alpha}\right)M\big(\tau,\Vert F\Vert_{\text{Lip-}(N+1)},\max_{1\leqslant i\leqslant N-1}\big|Y_{0}^{i}\big|,\max_{1\leqslant i\leqslant N-1}\big|\tilde{Y}_{0}^{i}\big|,\nonumber \\
 & \ \ \ \Vert\mathcal{Y}\Vert_{\mathbf{X};\alpha},\Vert\tilde{\mathcal{Y}}\Vert_{\tilde{\mathbf{X}};\alpha},\Vert\mathbf{X}\Vert_{\beta},\Vert\tilde{\mathbf{X}}\Vert_{\beta}\big)\cdot\big(d_{\mathbf{X},\mathbf{X};\alpha}\big(\mathcal{Y},\tilde{\mathcal{Y}}\big)+\rho_{\beta}({\bf X},\tilde{{\bf X}})\nonumber \\
 & \ \ \ +\max_{0\leqslant i\leqslant N-1}\big|Y_{0}^{i}-\tilde{Y}_{0}^{i}\big|\big).\label{eq:CompCtyEst}
\end{align}
Here $M$ is a continuous function that is increasing in every variable.
\end{lem}
\begin{rem}
The factor $\max(\tau^{\alpha},\tau^{\beta-\alpha})$ and the independence
of $Y_{0}^{0}$ in the function $M$ are both important for the patching
argument in the RDE context.
\end{rem}
\begin{proof}
Observe that ${\cal R}{\cal J}_{s,t}^{i}={\cal R}{\cal I}_{s,t}^{i}$
where ${\cal I}\triangleq\int_{0}^{\cdot}F({\cal Y})d{\bf X}.$ By
the integration estimate (cf. Lemma \ref{lem:Integration-Lemma})
with $\tilde{\mathcal{Z}}=0$ and $\tilde{\mathbf{X}}=\mathbf{X}$,
we have 
\begin{align}
\big\|{\cal J}\big\|_{{\bf X};\alpha} & \leqslant\max\big(\tau^{\alpha},\tau^{\beta-\alpha}\big)M\big(\tau,\Vert\mathbf{X}\Vert_{\beta}\big)\big(\Vert F\big(\mathcal{Y}\big)\Vert_{\mathbf{X};\alpha}+\big|F({\cal Y})_{0}^{N-1}\big|\big).\label{eq:InvarianceStep1}
\end{align}
Since $F(\mathcal{Y})_{0}^{N-1}$ can be expressed as a polynomial
of $Y_{0}^{1},\cdots,Y_{0}^{N-2}$ (cf (\ref{eq:GDerZ})), there is
a continuous increasing function $M$ such that 
\begin{equation}
\big|F({\cal Y})_{0}^{N-1}\big|\leqslant\Vert F\Vert_{\text{Lip-}N}M\big(\max_{1\leqslant i\leqslant N-1}\big|Y_{0}^{i}\big|\big).\label{eq:InitialN-1}
\end{equation}
 In addition, from Lemma \ref{lem:LipschitzLemma} we know that 
\begin{equation}
\Vert F\big(\mathcal{Y}\big)\Vert_{\mathbf{X};\alpha}\leqslant\Vert F\Vert_{\text{Lip-}N}M\big(\tau,\max_{1\leqslant i\leqslant N-1}\big|Y_{0}^{i}\big|,\big\|\mathcal{Y}\big\|_{{\bf X};\alpha},\Vert\mathbf{X}\Vert_{\alpha}\big).\label{eq:LipschitzSolution}
\end{equation}
Note that $\Vert\mathbf{X}\Vert_{\alpha}\leqslant M(T)\Vert\mathbf{X}\Vert_{\beta}$
for some increasing continuous function $M$. The inequality (\ref{eq:CompEst})
follows by putting the estimates (\ref{eq:InvarianceStep1}) and (\ref{eq:LipschitzSolution})
together.

For the continuity estimate, first note from Lemma \ref{lem:Integration-Lemma}
that

\begin{align}
 & d_{\mathbf{X},\tilde{{\bf X}};\alpha}\big(\mathcal{J},\tilde{\mathcal{J}}\big)\nonumber \\
 & =d_{\mathbf{X},\tilde{{\bf X}};\alpha}\big(\int_{0}^{\cdot}F(\mathcal{Y})\mathrm{d}{\bf X},\int_{0}^{\cdot}F(\tilde{\mathcal{Y}})\mathrm{d}\tilde{{\bf X}}\big)\nonumber \\
 & \leqslant\max\big(\tau^{\alpha},\tau^{\beta-\alpha}\big)M\big(\tau,\Vert\mathbf{X}\Vert_{\beta},\big\| F(\tilde{\mathcal{Y}})\big\|{}_{\tilde{{\bf X}};\alpha},\big|F(\tilde{{\cal Y}})_{0}^{N-1}\big|\big)\nonumber \\
 & \ \ \ \times\big(d_{\mathbf{X},\tilde{{\bf X}};\alpha}\big(F(\mathcal{Y}),F(\tilde{\mathcal{Y}})\big)+\rho_{\beta}({\bf X},\tilde{{\bf X}})+\big|F({\cal Y})_{0}^{N-1}-F(\tilde{{\cal Y}})_{0}^{N-1}\big|\big).\label{eq:CompCtyInt}
\end{align}
By (\ref{eq:InitialN-1}) and (\ref{eq:LipschitzSolution}), we have
\begin{equation}
\|F(\tilde{{\cal Y}})\|_{\tilde{{\bf X}};\alpha}\vee\big|F(\tilde{{\cal Y}})_{0}^{N-1}\big|\leqslant M\big(\tau,\|F\|_{\text{Lip-}N},\|\tilde{{\bf X}}\|_{\beta},\max_{1\leqslant i\leqslant N-1}\big|\tilde{Y}_{0}^{i}\big|,\|\tilde{{\cal Y}}\|_{\tilde{{\bf X}};\alpha}\big)\label{eq:CompCtyS0}
\end{equation}
with some continuous increasing function $M$. Moreover, from (\ref{eq:LipschitzDifference})
we have

\begin{align}
 & d_{\mathbf{X},\tilde{{\bf X}};\alpha}\big(F(\mathcal{Y}),F(\tilde{\mathcal{Y}})\big)\nonumber \\
 & \leqslant\Vert F\Vert_{\text{Lip-}N+1}M\big(\tau,\max_{1\leqslant i\leqslant N-1}\big|Y_{0}^{i}\big|,\max_{1\leqslant i\leqslant N-1}\big|\tilde{Y}_{0}^{i}\big|,\Vert\mathcal{Y}\Vert_{\mathbf{X};\alpha},\Vert\tilde{\mathcal{Y}}\Vert_{\tilde{\mathbf{X}};\alpha},\Vert\mathbf{X}\Vert_{\alpha},\Vert\tilde{\mathbf{X}}\Vert_{\alpha}\big)\nonumber \\
 & \ \ \ \times\big(d_{\mathbf{X},\tilde{{\bf X}};\alpha}\big(\mathcal{Y},\tilde{\mathcal{Y}}\big)+\max_{0\leqslant i\leqslant N-1}\big|Y_{0}^{i}-\tilde{Y}_{0}^{i}\big|+\rho_{\alpha}({\bf X},\tilde{{\bf X}})\big).\label{eq:CompCtyS1}
\end{align}
Also note from (\ref{eq:GDerZ}) that 
\begin{equation}
\big|F({\cal Y})_{0}^{N-1}-F(\tilde{{\cal Y}})_{0}^{N-1}\big|\leqslant M\big(\max_{1\leqslant i\leqslant N-1}\big|Y_{0}^{i}\big|\big)\cdot\Vert F\Vert_{\text{Lip-}N}\max_{0\leqslant i\leqslant N-1}\big|Y_{0}^{i}-\tilde{Y}_{0}^{i}\big|.\label{eq:CompCtyS2}
\end{equation}
The continuity estimate (\ref{eq:CompCtyEst}) follows by applying
(\ref{eq:CompCtyS0}) (\ref{eq:CompCtyS1}) and (\ref{eq:CompCtyS2})
to the inequality (\ref{eq:CompCtyInt}).
\end{proof}

\subsection{Existence, uniqueness and continuity of RDE solutions}

We first define the notion of solution for the RDE (\ref{eq:RDE}).
All paths are assumed to be defined on $[0,T]$ ($T>0$ is given fixed).
\begin{defn}
\label{def:DefinitionRDESolution} Let $\mathbf{X}$ be a $\beta$-Hölder
geometric rough path over $V$, and let $Y_{0}\in U$. We say that
${\cal Y}\in{\cal D}_{{\bf x};\alpha}(U)$ is a \textit{solution}
to the RDE (\ref{eq:RDE}) with initial condition $Y_{0}$, if 
\begin{align*}
Y_{t}^{0} & =Y_{0}+\big(\int_{0}^{\cdot}F\left(\mathcal{Y}\right)\mathrm{d}{\bf X}\big)_{t}^{0},\ Y_{t}^{i}=\big(\int_{0}^{\cdot}F\left(\mathcal{Y}\right)\mathrm{d}{\bf X}\big){}_{t}^{i}\qquad\text{for }1\leqslant i\leqslant N-1.
\end{align*}
\end{defn}
The main theorem in this part is stated as follows.
\begin{thm}
\label{thm:RDESol}(i) {[}Existence and uniqueness{]} Let $\mathbf{X}$
be a given $\beta$-Hölder geometric rough path over $V$. For each
$Y_{0}\in U$, there exists a unique solution ${\cal Y}\in{\cal D}_{{\bf X};\alpha}(U)$
to the RDE (\ref{eq:RDE}) in the sense of Definition \ref{def:DefinitionRDESolution}.\\
(ii) {[}Continuity estimate{]} Let $\mathbf{X}$ and $\tilde{\mathbf{X}}$
be $\beta$-Hölder geometric rough paths over $V$, and let $Y_{0},\tilde{Y}_{0}\in U$.
Suppose that 
\[
\|{\bf X}\|_{\beta}\vee\|\tilde{{\bf X}}\|_{\beta}\vee|Y_{0}|\vee|\tilde{Y}_{0}|\leqslant B
\]
with some constant $B>0.$ Let $\mathcal{Y}$ and $\tilde{\mathcal{Y}}$
be the solutions to (\ref{eq:RDE}) driven by $\mathbf{X}$ and $\tilde{\mathbf{X}}$
with initial conditions $Y_{0}$ and $\tilde{Y}_{0}$ respectively.
Then the following estimate holds true:
\begin{equation}
d_{\mathbf{X},\tilde{\mathbf{X}};\alpha}\big(\mathcal{Y},\tilde{\mathcal{Y}}\big)\leqslant M(T,B,\|F\|_{\text{Lip-}(N+1)})\big(\rho_{\beta}({\bf X},\tilde{{\bf X}})+\big|Y_{0}-\tilde{Y}_{0}\big|\big).\label{eq:CtyEstRDE}
\end{equation}
\end{thm}
The rest of this subsection is devoted to the proof of Theorem \ref{thm:RDESol}.

\subsubsection{Local contraction}

We shall prove existence and uniqueness by using the Banach fixed
point theorem. Note that the ``constants'' appearing in the rough
integration and Lipschitz transformation estimates depend on $\Vert\mathcal{Y}\Vert_{\mathbf{X};\alpha}$.
As a result, the mapping ${\cal M}:{\cal Y}\mapsto Y_{0}+\int F({\cal Y})d{\bf X}$
can only be a contraction if we restrict ${\cal M}$ on a bounded
subset, say a unit ball. To determine the center $\mathcal{W}$ (as
a controlled rough path) of such a ball, it is natural to require
$W_{0}^{0}=Y_{0}$ as this is the given initial condition. The higher
order terms $W^{i}$ ($i\geqslant1$) are chosen such that $\mathcal{R}\mathcal{W}_{s,t}^{i}=0$.
This is formulated precisely in the following lemma.
\begin{lem}
\label{lem:ConstantControlled}Let $Y_{0}\in U$ be given. We set
\begin{align*}
W_{0}^{0} & \triangleq Y_{0},\ W_{0}^{1}\triangleq F^{0}(Y_{0}),
\end{align*}
and inductively
\begin{equation}
W_{0}^{r+1}\triangleq\sum_{j=0}^{N-1}\frac{F^{j}(Y_{0})}{j!}\big(\sum_{i_{1}+\ldots+i_{j}=r}\big(W_{0}^{i_{1}}\boxtimes\ldots\boxtimes W_{0}^{i_{j}}\big)\circ\delta_{j}\big)\in{\cal L}\big(V^{\otimes(r+1)};U\big).\label{eq:InductiveW}
\end{equation}
Define the path $\mathcal{W}=(W^{0},W^{1},W^{2},\ldots,W^{N-1})$
by 
\[
W_{t}^{i}(\xi)=W_{0}^{i}(\xi)+W_{0}^{i+1}\big(X_{0,t}^{1}\otimes\xi\big)+\cdots+W_{0}^{N-1}\big(X_{0,t}^{N-1-i}\otimes\xi\big).
\]
Then $\mathcal{W}$ is a controlled rough path with respect to $\mathbf{X}$.
More specifically, we have ${\cal RW}^{i}\equiv0$ for each $0\leqslant i\leqslant N-1$.
\end{lem}
\begin{rem}
The initial value ${\cal W}_{0}$ is canonically determined by $Y_{0}$
and $F.$
\end{rem}
\begin{proof}
Note that 
\begin{align*}
W_{t}^{i}(\xi) & =\sum_{j=i}^{N-1}W_{0}^{j}\big(X_{0,t}^{j-i}\otimes\xi\big)=\sum_{j=i}^{N-1}W_{0}^{j}\big(\sum_{k=0}^{j-i}X_{0,s}^{j-i-k}\otimes X_{s,t}^{k}\otimes\xi\big)\\
 & =\sum_{k=0}^{N-1-i}\sum_{j=i+k}^{N-1}W_{0}^{j}\big(X_{0,s}^{j-i-k}\otimes X_{s,t}^{k}\otimes\xi\big)=\sum_{k=0}^{N-1-i}W_{s}^{i+k}\big(X_{s,t}^{k}\otimes\xi\big)
\end{align*}
As a result, we have $\mathcal{R}\mathcal{W}_{s,t}^{i}=0$ for all
$s\leqslant t$ and $0\leqslant i\leqslant N-1$.
\end{proof}
The following lemma gives the local existence and uniqueness for the
RDE (\ref{eq:RDE}).
\begin{lem}
\label{lem:LocalExistenceAndUniqueness}Given $\tau>0$, let 
\[
\mathcal{B}_{\tau}\triangleq\big\{\mathcal{Y}\in\mathcal{D}_{\mathbf{X};\alpha}(U):\|{\cal Y}-{\cal W}\|_{{\bf X};\alpha}\leqslant1,\mathcal{Y}_{0}=\mathcal{W}_{0}\big\}
\]
Then there exists $\tau>0$ , which is independent of $Y_{0}$ and
depends only on $\alpha,\beta,\mathbf{X}$ and $\Vert F\Vert_{\text{Lip-}(N+1)}$,
such that:

\vspace{2mm}\noindent (i) the mapping ${\cal M}:{\cal Y}\mapsto{\cal J}$
sends $\mathcal{B}_{\tau}$ to $\mathcal{B}_{\tau}$, where ${\cal J}$
is the controlled rough path defined by (\ref{eq:JPath});\\
(ii) the mapping $\mathcal{M}$ is a contraction on $\mathcal{B}_{\tau}$
with respect to the norm $\interleave\cdot\interleave_{{\bf X};\alpha}$
defined by (\ref{eq:ControlNorm}).\\
(iii) The RDE (\ref{eq:RDE}) has a unique solution $\mathcal{Y}$
on $[0,\tau]$ satisfying
\begin{equation}
\|{\cal Y}\|_{{\bf X};\alpha}\leqslant1.\label{eq:BoundOnSolution}
\end{equation}
\end{lem}
\begin{proof}
(i) We first prove by induction that $W_{0}^{i}=J_{0}^{i}$ for all
$i$. The $i=0,1$ cases follow directly from the definition of $\mathcal{J}$
and $\mathcal{W}$. For the induction step, note that by the definition
of $\mathcal{J}$, if $\xi\in V^{\otimes\left(r+1\right)}$, then
\begin{align*}
J_{0}^{r+1}\left(\xi\right) & =\left(F\left({\cal Y}\right)\right)_{0}^{r}(\xi)\\
 & =\sum_{j=0}^{N-1}\frac{F^{j}(Y_{0})}{j!}\sum_{i_{1}+\cdots+i_{j}=r}\big(Y_{0}^{i_{1}}\boxtimes\cdots\boxtimes Y_{0}^{i_{j}}\big)\circ\delta_{j}(\mathcal{\xi})\\
 & =\sum_{j=0}^{N-1}\frac{F^{j}(Y_{0})}{j!}\sum_{i_{1}+\cdots+i_{j}=r}\big(W_{0}^{i_{1}}\boxtimes\cdots\boxtimes W_{0}^{i_{j}}\big)\circ\delta_{j}(\mathcal{\xi})\qquad(\text{since \ensuremath{{\cal Y}_{0}={\cal W}_{0}}})\\
 & =W_{0}^{r+1}\qquad(\text{by definition of }W_{0}^{r+1}).
\end{align*}
Therefore, ${\cal W}_{0}={\cal J}_{0}.$

Next, we recall from (\ref{eq:CompEst}) that
\[
\big\|{\cal J}\big\|_{{\bf X};\alpha}\leqslant\max\big(\tau^{\alpha},\tau^{\beta-\alpha}\big)M\big(\tau,\Vert\mathbf{X}\Vert_{\beta},\max_{1\leqslant i\leqslant N-1}\big|Y_{0}^{i}\big|,\big\|\mathcal{Y}\big\|{}_{\mathbf{X};\alpha},\Vert F\Vert_{\text{Lip-}N}\big).
\]
Since $\mathcal{R}\mathcal{W}_{s,t}^{i}=0$ by Lemma \ref{lem:ConstantControlled}
and ${\cal Y}\in{\cal B}_{\tau}$, we know that 
\[
\|{\cal Y}\|_{{\bf X};\alpha}=d_{\mathbf{X},\mathbf{X};\alpha}\big(\mathcal{Y},\mathcal{W}\big)\leqslant1.
\]
The inductive definition of $\mathcal{W}_{0}$ in (\ref{eq:InductiveW})
implies that there is a continuous increasing function $M$ such that
for $1\leqslant i\leqslant N-1$ 
\begin{equation}
\big|Y_{0}^{i}\big|=\big|W_{0}^{i}\big|\leqslant M\big(\Vert F\Vert_{\text{Lip-}(N-1)}\big).\label{eq:Y_0Bound}
\end{equation}
As a result, we can choose $\tau$ to be sufficiently small (depending
on $\|{\bf X}\|_{\beta}$ and $\|F\|_{\text{Lip-}N}$), such that
\begin{equation}
\max\big(\tau^{\alpha},\tau^{\beta-\alpha}\big)M\big(\tau,\Vert\mathbf{X}\Vert_{\beta},\Vert F\Vert_{\text{Lip-}N}\big)<1.\label{eq:ChoosingT}
\end{equation}
This ensures that ${\cal J}=\mathcal{M}(\mathcal{Y})\in\mathcal{B}_{\tau}$.
Note that the choice of $\tau$ is independent of $Y_{0}^{0}$.

(ii) Let $\mathcal{Y},\tilde{\mathcal{Y}}\in\mathcal{B}_{\tau}$.
By using Lemma \ref{lem:ContractionEstimates} and that $\mathcal{Y}_{0}=\tilde{\mathcal{Y}}_{0}$,
we have

\begin{align*}
d_{\mathbf{X},\mathbf{X};\alpha}\big(\mathcal{M}(\mathcal{Y}),\mathcal{M}(\tilde{\mathcal{Y}})\big)\leqslant & \max\big(\tau^{\alpha},\tau^{\beta-\alpha}\big)M\big(\tau,\Vert F\|_{\text{Lip-}(N+1)},\max_{1\leqslant i\leqslant N-1}\big|Y_{0}^{i}\big|,\\
 & \big\|\mathcal{Y}\big\|{}_{\mathbf{X};\alpha},\big\|\tilde{\mathcal{Y}}\big\|{}_{\mathbf{X};\alpha},\Vert\mathbf{X}\Vert_{\beta}\big)d_{\mathbf{X},\mathbf{X};\alpha}\big(\mathcal{Y},\tilde{\mathcal{Y}}\big).
\end{align*}
According to (\ref{eq:Y_0Bound}) and the fact that ${\cal Y},\tilde{{\cal Y}}\in{\cal B}_{\tau},$
we may further choose $\tau$ such that
\[
\max\big(\tau^{\alpha},\tau^{\beta-\alpha}\big)M\big(\tau,\Vert F\|_{\text{Lip-}(N+1)},\max_{1\leqslant i\leqslant N-1}\big|Y_{0}^{i}\big|,\big\|\mathcal{Y}\big\|{}_{\mathbf{X};\alpha},\big\|\tilde{\mathcal{Y}}\big\|{}_{\mathbf{X};\alpha},\Vert\mathbf{X}\Vert_{\beta}\big)<\frac{1}{2}.
\]
Therefore, we have
\[
\interleave{\cal M}({\cal Y})-{\cal M}(\tilde{{\cal Y}})\interleave_{{\bf X};\alpha}\leqslant\frac{1}{2}\interleave{\cal Y}-\tilde{{\cal Y}}\interleave_{{\bf X};\alpha},
\]
which shows that the mapping ${\cal M}:{\cal B}_{\tau}\rightarrow{\cal B}_{\tau}$
is a contraction for such choice of $\tau$.

(iii) Let $\tau$ be chosen as in Part (ii). Note that a solution
to the RDE (\ref{eq:RDE}) is a fixed point of the mapping $\mathcal{M}$.
Since ${\cal B}_{\tau}$ is a closed subset of the Banach space $({\cal D}_{{\bf X},\alpha}(U),\interleave\cdot\interleave_{{\bf X};\alpha}),$
by Part (ii) and the Banach fixed point theorem, we know that the
RDE (\ref{eq:RDE}) admits a unique solution ${\cal Y}\in{\cal B}_{\tau}$
as a controlled rough path on $[0,\tau].$ The inequality (\ref{eq:BoundOnSolution})
is just a consequence of $\mathcal{Y}\in\mathcal{B}_{\tau}$.
\end{proof}
\begin{rem}
\label{rem:Y0DetYi}It is interesting to point out that, if ${\cal Y}=(Y^{0},\cdots,Y^{N-1})$
is a solution to the RDE (\ref{eq:RDE}), then at each $t$ the values
$Y_{t}^{i}$ ($1\leqslant i\leqslant N-1$) are all canonically determined
by the value $Y_{t}^{0}$ of the $0$-th level path. Indeed, by Definition
\ref{def:DefinitionRDESolution} we have $Y_{t}^{i}=F({\cal Y})_{t}^{i-1}$
for all $i\geqslant1.$ The determination of $Y_{t}^{i}$ from $Y_{t}^{0}$
is through the same relation as (\ref{eq:InductiveW}). This observation
in the later patching argument.
\end{rem}

\subsubsection{A patching lemma}

In order to obtain global existence, we need to patch local solutions
in the sense controlled rough paths. The lemma below justifies the
patching of controlled rough paths in general.
\begin{lem}
\label{lem:PatchingOfControlledNorm} (i) Let ${\bf X}$ be a geometric
rough path on $[a,b]$ and let $u\in(a,b)$ be fixed. Let ${\cal Y}$
be a continuous path on $[a,b]$ such that ${\cal Y}|_{[a,u]}$ (respectively,
${\cal Y}|_{[u,b]}$) is controlled by ${\bf X}|_{[a,u]}$ (respectively,
by ${\bf X}|_{[u,b]}$). Then ${\cal Y}$ is controlled by ${\bf X}$
on $[a,b]$.\\
(ii) Let ${\bf X}$, $\tilde{{\bf X}}$ be $\beta$-H\"older geometric
rough paths on $[a,b]$ and let ${\cal Y}$, $\tilde{{\cal Y}}$ be
$\alpha$-H\"older controlled rough paths with respect to ${\bf X}$,
$\tilde{{\bf X}}$ respectively. Let $u\in(a,b)$ be fixed. Then we
have
\begin{align}
d_{\mathbf{X},\tilde{\mathbf{X}};\alpha}\big(\mathcal{Y},\tilde{\mathcal{Y}}\big) & \leqslant d_{\mathbf{X},\tilde{\mathbf{X}};\alpha}\big(\mathcal{Y}|_{[u,b]},\tilde{\mathcal{Y}}|_{[u,b]}\big)+d_{\mathbf{X},\tilde{\mathbf{X}};\alpha}\big(\mathcal{Y}|_{[a,u]},\tilde{\mathcal{Y}}|_{[a,u]}\big)\big(1+\Vert\mathbf{X}\Vert_{\alpha}\big)\nonumber \\
 & \qquad+\Vert\tilde{\mathcal{Y}}|_{[a,u]}\Vert_{\tilde{\mathbf{X}};\alpha}\rho_{\alpha}({\bf X},\tilde{{\bf X}}).\label{eq:CtyPatch}
\end{align}
\end{lem}
\begin{proof}
(i) It is enough to consider the remainder ${\cal RY}_{s,t}^{k}$
when $s<u<t.$ Note that 
\begin{align*}
\sum_{i=k}^{N-1}Y_{s}^{i}X_{s,t}^{i-k} & =\sum_{i=k}^{N-1}Y_{s}^{i}\sum_{j=0}^{i-k}X_{s,u}^{i-k-j}X_{u,t}^{j}=\sum_{i=k}^{N-1}Y_{s}^{i}\sum_{j=k}^{i}X_{s,u}^{i-j}X_{u,t}^{j-k}\\
 & =\sum_{j=k}^{N-1}\big(\sum_{i=j}^{N-1}Y_{s}^{i}X_{s,u}^{i-j}\big)X_{u,t}^{j-k}\\
 & =\sum_{j=k}^{N-1}Y_{u}^{j}X_{u,t}^{j-k}-\sum_{j=k}^{N-1}\big(Y_{u}^{j}-\sum_{i=j}^{N-1}Y_{s}^{i}X_{s,u}^{i-j}\big)X_{u,t}^{j-k}.
\end{align*}
Therefore, 
\begin{align*}
Y_{t}^{k}-\sum_{i=k}^{N-1}Y_{s}^{i}X_{s,t}^{i-k} & =Y_{t}^{k}-\sum_{j=k}^{N-1}Y_{u}^{j}X_{u,t}^{j-k}+\sum_{j=k}^{N-1}\big(Y_{u}^{j}-\sum_{i=j}^{N-1}Y_{s}^{i}X_{s,u}^{i-j}\big)X_{u,t}^{j-k},
\end{align*}
or equivalently 
\begin{equation}
\mathcal{R}\mathcal{Y}_{s,t}^{k}=\mathcal{R}\mathcal{Y}_{u,t}^{k}+\sum_{j=k}^{N-1}\mathcal{R}\mathcal{Y}_{s,u}^{j}X_{u,t}^{j-k}.\label{eq:ControlledConcatenation}
\end{equation}
From (\ref{eq:ControlledConcatenation}), it is clear that the H\"oler
regularity of ${\cal R}{\cal Y}_{s,t}^{k}$ is $|t-s|^{(N-k)\alpha}$.

(ii) According to (\ref{eq:ControlledConcatenation}), for $s<u<t$
we also have 
\begin{align*}
 & \big|\mathcal{R}\mathcal{Y}_{s,t}^{k}-\mathcal{R}\tilde{\mathcal{Y}}_{s,t}^{k}\big|\\
 & \leqslant\big|\mathcal{R}\mathcal{Y}_{u,t}^{k}-\mathcal{R}\tilde{\mathcal{Y}}_{u,t}^{k}\big|+\sum_{j=k}^{N-1}\big|\mathcal{R}\mathcal{Y}_{s,u}^{j}-\mathcal{R}\tilde{\mathcal{Y}}_{s,u}^{j}\big|\|X^{j-k}\|_{\alpha}\left(t-u\right)^{(j-k)\alpha}\\
 & \ \ \ +\sum_{j=k}^{N-1}\big|\mathcal{R}\tilde{\mathcal{Y}}_{s,u}^{j}\big|\Vert X^{j-k}-\tilde{X}^{j-k}\Vert_{\alpha}(t-u)^{(j-k)\alpha}\\
 & \leqslant d_{\mathbf{X},\tilde{\mathbf{X}};\alpha}\big(\mathcal{Y}|_{[u,b]},\tilde{\mathcal{Y}}|_{[u,b]}\big)(t-u)^{(N-k)\alpha}\\
 & \ \ \ +\sum_{j=k}^{N-1}\Vert\mathcal{R}\mathcal{Y}^{j}|_{[a,u]}-\mathcal{R}\tilde{\mathcal{Y}}^{j}|_{[a,u]}\Vert_{(N-j)\alpha}(u-s)^{(N-j)\alpha}\Vert\mathbf{X}\Vert_{\alpha}(t-u)^{(j-k)\alpha}\\
 & \ \ \ +\sum_{j=k}^{N-1}\Vert\mathcal{R}\tilde{\mathcal{Y}}^{j}|_{[a,u]}\Vert_{(N-j)\alpha}\rho_{\alpha}({\bf X},\tilde{{\bf X}})(u-s)^{(N-j)\alpha}(t-u)^{(j-k)\alpha}.
\end{align*}
The inequality (\ref{eq:CtyPatch}) thus follows.
\end{proof}

\subsubsection{Global existence, uniqueness and continuity}

By patching local solutions and local estimates, we are now able to
establish the global well-posedness of the RDE (\ref{eq:RDE}) in
the space of controlled rough paths. Let $\alpha,\beta,N,F$ be given
as before.

\begin{proof}[Proof of Theorem \ref{thm:RDESol}]\textit{Existence}.
Let $\tau$ be given by Lemma \ref{lem:LocalExistenceAndUniqueness}.
According to Lemma \ref{lem:LocalExistenceAndUniqueness}, we have
a solution $\mathcal{Y}[1]$ on $[0,\tau]$ satisfying 
\begin{align*}
Y[1]_{t}^{0} & =Y_{0}+\int_{0}^{t}F(Y[1])\mathrm{d}X,\ Y[1]_{t}^{i}=\big[F(\mathcal{Y}[1])\big]_{t}^{i-1}\qquad\forall t\in[0,\tau].
\end{align*}
We define a sequence of controlled paths $\{{\cal Y}[n]:n\geqslant1\}$
on $[0,\tau]$ inductively in the following way. By applying Lemma
\ref{lem:LocalExistenceAndUniqueness} with $Y_{0}={\cal Y}[n-1]_{\tau}$
and $\mathbf{X}_{t}={\bf X}_{(n-1)\tau+t}$, we obtain a controlled
rough path $\mathcal{Y}[n]$ on $[0,\tau]$ satisfying 
\begin{align*}
Y[n]_{t}^{0} & =Y[n-1]_{\tau}^{0}+\int_{0}^{t}F(Y[n])\mathrm{d}X,\ Y[n]_{t}^{i}=\big[F({\cal Y}[n])\big]_{t}^{i-1}\ \ \ \forall t\in[0,\tau].
\end{align*}
We now define $\mathcal{Y}=(Y^{0},\ldots,Y^{N-1})$ as a path on $[0,\infty)$
by concatenating all the ${\cal Y}[n]$'s, namely
\[
\mathcal{Y}_{(n-1)\tau+t}=\mathcal{Y}[n]_{t},\qquad t\in[0,\tau].
\]
Note from Remark \ref{rem:Y0DetYi} that ${\cal Y}$ is well defined.
By Lemma \ref{lem:PatchingOfControlledNorm}, $\mathcal{Y}$ is a
controlled rough path with respect to ${\bf X}$.

For any $t\geqslant0,$ if $t\in[(n-1)\tau,n\tau]$ we have 
\begin{align*}
Y_{t}^{i} & =Y[n]_{t-(n-1)\tau}^{i}=F({\cal Y}[n])_{t-(n-1)\tau}^{i-1}=F({\cal Y})_{t}^{i-1}.
\end{align*}
It remains to show that, 
\begin{equation}
Y_{t}^{0}=Y_{0}+\int_{0}^{t}F(Y)\mathrm{d}X\qquad\forall t\geqslant0.\label{eq:0Level}
\end{equation}
We use induction on $n$. If $t\in[(n-1)\tau,n\tau]$, then
\begin{align*}
Y_{t}^{0} & =Y[n-1]_{t-(n-1)T}^{0}\\
 & =Y[n-1]_{0}^{0}+\int_{0}^{t-(n-1)\tau}F(Y[n-1]_{\cdot})\mathrm{d}X_{\cdot+(n-1)\tau}\\
 & =Y_{0}+\int_{0}^{(n-1)\tau}F(Y)\mathrm{d}X+\int_{0}^{t-(n-1)\tau}F(Y[n-1]_{\cdot})\mathrm{d}X_{\cdot+(n-1)\tau}\\
 & =Y_{0}+\int_{0}^{(n-1)\tau}F(Y)\mathrm{d}X+\int_{(n-1)\tau}^{t}F(Y)\mathrm{d}X\\
 & =Y_{0}+\int_{0}^{t}F(Y)\mathrm{d}X,
\end{align*}
where the third equality follows from induction hypothesis. Therefore,
(\ref{eq:0Level}) holds. We have thus obtained the existence of solution
on $[0,\infty)$.

\textit{Uniqueness}. Let $\mathcal{Y}^{1}$ and $\mathcal{Y}^{2}$
be two solutions to the RDE (\ref{eq:RDE}). Suppose that
\[
\sigma=\sup\{t\in[0,\infty):\mathcal{Y}_{s}^{1}=\mathcal{Y}_{s}^{2}\quad\forall s\in[0,t]\}<\infty
\]
Then $\mathcal{Y}_{\sigma}^{1}=\mathcal{Y}_{\sigma}^{2}=(Y_{\sigma}^{0},\ldots,Y_{\sigma}^{N-1})$.
According to Lemma \ref{lem:ContractionEstimates} with ${\bf X}=\tilde{{\bf X}}$,
for all $\tau$ sufficiently small, we have
\begin{align*}
 & d_{\mathbf{X},\mathbf{X};\alpha}\big(\mathcal{Y}^{1}|_{[\sigma,\sigma+\tau]},\mathcal{Y}^{2}|_{[\sigma,\sigma+\tau]}\big)\\
 & \leqslant\max\big(\tau^{\alpha},\tau^{\beta-\alpha}\big)M\big(\tau,\Vert F\Vert_{\text{Lip}_{N+1}},\max_{1\leqslant i\leqslant N-1}\big|(Y^{1})_{\sigma}^{i}\big|,\max_{1\leqslant i\leqslant N-1}\big|(Y^{2})_{\sigma}^{i}\big|,\\
 & \ \ \ \Vert\mathcal{Y}^{1}\Vert_{\mathbf{X};\alpha},\Vert\mathcal{Y}^{2}\Vert_{\mathbf{X};\alpha},\Vert\mathbf{X}\Vert_{\beta}\big)d_{\mathbf{X},\mathbf{X};\alpha}\big(\mathcal{Y}^{1}|_{[\sigma,\sigma+\tau]},\mathcal{Y}^{2}|_{[\sigma,\sigma+\tau]}\big).
\end{align*}
 If we choose $\tau$ to be such that 
\begin{align*}
 & \max\big(\tau^{\alpha},\tau^{\beta-\alpha}\big)M\big(\tau,\Vert F\Vert_{\text{Lip}_{N+1}},\max_{1\leqslant i\leqslant N-1}\big|(Y^{1})_{\sigma}^{i}\big|,\max_{1\leqslant i\leqslant N-1}\big|(Y^{2})_{\sigma}^{i}\big|,\\
 & \ \ \ \ \ \Vert\mathcal{Y}^{1}|_{[\sigma,\sigma+\tau]}\Vert_{\mathbf{X};\alpha},\Vert\mathcal{Y}^{2}|_{[\sigma,\sigma+\tau]}\Vert_{\mathbf{X};\alpha},\Vert\mathbf{X}\Vert_{\beta}\big)<1,
\end{align*}
then 
\[
d_{\mathbf{X},\mathbf{X};\alpha}\big(\mathcal{Y}^{1}|_{[\sigma,\sigma+\tau]},\mathcal{Y}^{2}|_{[\sigma,\sigma+\tau]}\big)=0.
\]
Since ${\cal Y}_{\sigma}^{1}={\cal Y}_{\sigma}^{2}$, this implies
that $\mathcal{Y}^{1}|_{[\sigma,\sigma+\tau]}=\mathcal{Y}^{2}|_{[\sigma,\sigma+\tau]}$
, which contradicts the definition of $\sigma.$ Therefore, ${\cal Y}^{1}={\cal Y}^{2}$
on $[0,\infty).$

\textit{Continuity estimate}. We now assume that all underlying paths
are defined on a given fixed interval $[0,T].$ To establish (\ref{eq:CtyEstRDE}),
it is equivalent to showing that 
\begin{equation}
d_{\mathbf{X},\tilde{\mathbf{X}};\alpha}\big(\mathcal{Y},\tilde{\mathcal{Y}}\big)\leqslant M\big(T,\Vert F\Vert_{\text{Lip-}(N+1)},B\big)\big(\sum_{i=0}^{N-1}\big|Y_{0}^{i}-\tilde{Y}_{0}^{i}\big|+\rho_{\beta}({\bf X},\tilde{{\bf X}})\big).\label{eq:PreCty}
\end{equation}
Indeed, observe that
\[
\sum_{i=0}^{N-1}\big|Y_{0}^{i}-\tilde{Y}_{0}^{i}\big|\leqslant M(\|F\|_{\text{Lip-}(N)})\big|Y_{0}^{0}-\tilde{Y}_{0}^{0}\big|,
\]
which is clear since all the $Y_{0}^{i}$'s and $\tilde{Y}_{0}^{i}$'s
are canonically determined by $Y_{0}^{0}$ and $\tilde{Y}_{0}^{0}$
via the relation (\ref{eq:InductiveW}). Therefore, (\ref{eq:CtyEstRDE})
follows from (\ref{eq:PreCty}).

To establish (\ref{eq:PreCty}), according to Lemma \ref{lem:ContractionEstimates}
and the fact that ${\cal Y},\tilde{{\cal Y}}$ are RDE solutions,
for any $\tau>0$ we have 
\begin{align*}
d_{\mathbf{X},\tilde{\mathbf{X}};\alpha}\big(\mathcal{Y},\tilde{\mathcal{Y}}\big)\leqslant & \max\big(\tau^{\alpha},\tau^{\beta-\alpha}\big)M\big(\tau,\Vert F\Vert_{\text{Lip-}(N+1)},\max_{1\leqslant i\leqslant N-1}\big|Y_{0}^{i}\big|,\max_{1\leqslant i\leqslant N-1}\big|\tilde{Y}_{0}^{i}\big|,\\
 & \Vert\mathcal{Y}\Vert_{\mathbf{X};\alpha},\Vert\tilde{\mathcal{Y}}\Vert_{\mathbf{X};\alpha},\Vert\mathbf{X}\Vert_{\beta},\Vert\tilde{\mathbf{X}}\Vert_{\beta}\big)\big(d_{\mathbf{X},\tilde{\mathbf{X}};\alpha}\big(\mathcal{Y},\tilde{\mathcal{Y}}\big)+\rho_{\beta}({\bf X},\tilde{{\bf X}})\\
 & +\max_{0\leqslant i\leqslant N-1}\Vert Y_{0}^{i}-\tilde{Y}_{0}^{i}\Vert\big)
\end{align*}
when restricted on $[0,\tau].$ Let $\tau_{B}$ be given by (\ref{eq:ChoosingT})
with $\Vert\mathbf{X}\Vert_{\beta}$ replaced by $B$. From (\ref{eq:Y_0Bound})
we have 
\begin{equation}
\max_{1\leqslant i\leqslant N-1}\big|Y_{0}^{i}\big|\vee\big|\tilde{Y}_{0}^{i}\big|\leqslant M\big(\Vert F\Vert_{\text{Lip-}(N-1)}\big)\label{eq:Y_0NormAgain}
\end{equation}
and from (\ref{eq:BoundOnSolution}) we also have (restricted on $[0,\tau_{B}]$)
\begin{equation}
\Vert{\cal Y}]\|_{{\bf X};\alpha}\vee\Vert\tilde{\mathcal{Y}}\Vert_{\mathbf{X};\alpha}\leqslant1.\label{eq:GubNorm}
\end{equation}
We choose $\tau$ to be sufficiently small, so that $\tau<\tau_{B}$,
$T$ is an integer multiple of $\tau$ and
\begin{align*}
\delta & \triangleq\max\big(\tau^{\alpha},\tau^{\beta-\alpha}\big)\times M\big(\tau,\Vert F\Vert_{\text{Lip}_{N+1}},M\big(\Vert F\Vert_{\text{Lip}_{N-1}}\big),\\
 & \ \ \ \ \ M\big(\Vert F\Vert_{\text{Lip}_{N-1}}\big),1,1,B,B\big)<1.
\end{align*}
It is important to note that $\tau$ is independent of $Y_{0}$ and
$\tilde{Y}_{0}$. It follows that 
\begin{equation}
d_{\mathbf{X},\tilde{\mathbf{X}};\alpha}\big(\mathcal{Y},\tilde{\mathcal{Y}}\big)\leqslant\frac{\delta}{1-\delta}\big(\rho_{\beta}({\bf X},\tilde{{\bf X}})+\max_{0\leqslant i\leqslant N-1}\big|Y_{0}^{i}-\tilde{Y}_{0}^{i}\big|\big)\label{eq:CtyEst1stInt}
\end{equation}
on $[0,\tau]$. By using (\ref{eq:CtyEst1stInt}) on each sub-interval
$[(n-1)\tau,n\tau]$ ($1\leqslant n\leqslant T/\tau$), we arrive
at 
\begin{equation}
d_{\mathbf{X},\tilde{\mathbf{X}};\alpha}\big(\mathcal{Y},\tilde{\mathcal{Y}}\big)\leqslant\frac{\delta}{1-\delta}\big(\rho_{\beta}({\bf X},\tilde{{\bf X}})+\max_{0\leqslant i\leqslant N-1}\big|Y_{(n-1)\tau}^{i}-\tilde{Y}_{(n-1)\tau}^{i}\big|\big)\label{eq:CtySubInt}
\end{equation}
when restricted on $[(n-1)\tau,n\tau]$.

From (\ref{eq:GubNorm}) it is clear that
\[
\Vert{\cal Y}|_{[(n-1)\tau,n\tau]}\Vert_{{\bf X};\alpha}\vee\Vert\tilde{\mathcal{Y}}|_{[(n-1)\tau,n\tau]}\Vert_{\tilde{{\bf X}};\alpha}\leqslant1\ \ \ \forall n.
\]
In addition, according to Lemma \ref{lem: estimating delta Y_alpha},
for each $0\leqslant i\leqslant N-2$ we have (restricted on $[0,\tau]$)
\begin{align*}
\big|Y_{\tau}^{i}-\tilde{Y}_{\tau}^{i}\big|\leqslant & \big|Y_{0}^{i}-\tilde{Y}_{0}^{i}\big|+\tau^{\alpha}\Vert Y^{i}-\tilde{Y}^{i}\Vert_{\alpha}\\
\leqslant & \tau^{\alpha}M\big(\tau,\Vert\mathbf{X}\Vert_{\alpha},\Vert\tilde{\mathbf{X}}\Vert_{\alpha},\max_{i+1\leqslant j\leqslant N-1}\big|Y_{0}^{j}\big|,\max_{i+1\leqslant j\leqslant N-1}\Vert\mathcal{R}\mathcal{Y}^{N-j}\Vert_{j\alpha}\big)\\
 & \times\big(\rho_{\alpha}({\bf X},\tilde{{\bf X}})+\sum_{i=0}^{N-1}\big|Y_{0}^{i}-\tilde{Y}_{0}^{i}\big|+d_{\mathbf{X},\tilde{\mathbf{X}};\alpha}\big(\mathcal{Y},\tilde{\mathcal{Y}}\big)\big).
\end{align*}
In view of (\ref{eq:Y_0NormAgain}) and (\ref{eq:GubNorm}), we can
further write 
\begin{equation}
\big|Y_{\tau}^{i}-\tilde{Y}_{\tau}^{i}\big|\leqslant M\big(\tau,\Vert F\Vert_{\text{Lip-}(N+1)},B\big)\big(\sum_{i=0}^{N-1}\big|Y_{0}^{i}-\tilde{Y}_{0}^{i}\big|+\rho_{\beta}({\bf X},\tilde{{\bf X}})\big).\label{eq:OneStepHolderNorm-1}
\end{equation}
By applying (\ref{eq:OneStepHolderNorm-1}) iteratively, we have 
\begin{equation}
\big|Y_{n\tau}^{i}-\tilde{Y}_{n\tau}^{i}\big|\leqslant M_{n}\big(\tau,\Vert F\Vert_{\text{Lip-}(N+1)},B\big)\big(\sum_{i=0}^{N-1}\big|Y_{0}^{i}-\tilde{Y}_{0}^{i}\big|+\rho_{\beta}({\bf X},\tilde{{\bf X}})\big)\label{eq:YntauDif}
\end{equation}
for all $n,$ where the increasing function $M_{n}$ can depend on
$n$.

To proceed further, we show by induction that 
\begin{equation}
d_{\mathbf{X},\tilde{\mathbf{X}};\alpha}\big(\mathcal{Y}|_{[0,n\tau]},\tilde{\mathcal{Y}}|_{[0,n\tau]}\big)\leqslant M_{n}\big(\tau,\Vert F\Vert_{\text{Lip-}(N+1)},B\big)\big(\sum_{i=0}^{N-1}\big|Y_{0}^{i}-\tilde{Y}_{0}^{i}\big|+\rho_{\beta}({\bf X},\tilde{{\bf X}})\big)\label{eq:Estntau}
\end{equation}
for each $1\leqslant n\leqslant T/\tau.$ Suppose that (\ref{eq:Estntau})
is true on $[0,(n-1)\tau]$. According to (\ref{eq:CtySubInt}) and
(\ref{eq:YntauDif}), we have
\begin{align}
 & d_{\mathbf{X},\tilde{\mathbf{X}};\alpha}\big(\mathcal{Y}|_{[(n-1)\tau,n\tau]},\tilde{\mathcal{Y}}|_{[(n-1)\tau,n\tau]}\big)\nonumber \\
 & \leqslant M\big(\tau,\Vert F\Vert_{\text{Lip-}(N+1)},B\big)\rho_{\beta}({\bf X},\tilde{{\bf X}})\nonumber \\
 & \ \ \ +M_{n-1}\big(\tau,\Vert F\Vert_{\text{Lip-}(N+1)},B\big)\big(\sum_{i=0}^{N-1}\big|Y_{0}^{i}-\tilde{Y}_{0}^{i}\big|+\rho_{\beta}({\bf X},\tilde{{\bf X}})\big)\nonumber \\
 & \leqslant M_{n}\big(\tau,\Vert F\Vert_{\text{Lip-}(N+1)},B\big)\big(\sum_{i=0}^{N-1}\big|Y_{0}^{i}-\tilde{Y}_{0}^{i}\big|+\rho_{\beta}({\bf X},\tilde{{\bf X}})\big).\label{eq:BddlastInt}
\end{align}
We can then apply Lemma \ref{lem:PatchingOfControlledNorm} to patch
the estimate on $[0,(n-1)\tau]$ with the one on $[(n-1)\tau,n\tau]$
given by (\ref{eq:BddlastInt}). This completes the induction step.
The estimate (\ref{eq:PreCty}) follows by taking $n=T/\tau$.

Now the proof of Theorem \ref{thm:RDESol} is complete.

\end{proof}
\begin{rem}
If the vector field $F$ and its derivatives are not uniformly bounded,
the solution to the RDE (\ref{eq:RDE}) may explode in finite time.
Similar discussion gives existence and uniqueness up to the explosion
time.
\end{rem}
\begin{rem}
In the continuity estimate (\ref{eq:CtyEstRDE}), it is possible to
also take into account the perturbation of the vector field $F$.
In this case, an extra term of $\|F-\tilde{F}\|_{\text{Lip-}(N+1)}$
will be included on the right hand side of (\ref{eq:CtyEstRDE}).
This extension is routine and for the sake of conciseness we do not
provide the details.
\end{rem}

\end{document}